\documentclass{amsart}
\usepackage{xy, enumerate, pstcol, amssymb, amsfonts, amsbsy, amsthm, amsmath, amscd, epsfig, latexsym, mathtools, stmaryrd, epic, eepic, eucal, multicol, fancyhdr, graphicx}
\usepackage{hyperref}
\hypersetup{
	colorlinks=true,
	linkcolor=blue,
	filecolor=magenta,      
	urlcolor=cyan
}

\DeclareMathAlphabet{\mathpzc}{OT1}{pzc}{m}{it}

\newtheorem{theorem}{\bf Theorem}[section]
\newtheorem{maintheorem}{\bf Main Theorem}
\newtheorem{lemma}[theorem]{\bf Lemma}
\newtheorem{propos}[theorem]{\bf Proposition}
\newtheorem{corol}[theorem]{\bf Corollary}

\theoremstyle{definition}
\newtheorem{defi}[theorem]{\bf Definition}
\newtheorem{rmk}[theorem]{\bf Remark}

\newcommand{\A}{\mathbb A}

\newcommand{\Ecal}{\mathcal E}
\newcommand{\Fcal}{\mathcal F}

\newcommand{\Hcal}{\mathcal H}
\newcommand{\Ical}{\mathcal I}

\newcommand{\Lcal}{\mathcal L}
\newcommand{\Mcal}{\mathcal M}
\newcommand{\N}{\mathbb N}

\newcommand{\Ocal}{\mathcal O}
\newcommand{\Pro}{\mathbb P}
\newcommand{\Pcal}{\mathcal P}

\newcommand{\Qcal}{\mathcal Q}

\newcommand{\Scal}{\mathcal S}

\newcommand{\Vcal}{\mathcal V}
\newcommand{\Wcal}{\mathcal W}
\newcommand{\Xcal}{\mathcal X}

\newcommand{\Z}{\mathbb Z}
\newcommand{\Zcal}{\mathcal Z}

\newcommand{\Hom}{\text{Hom}}

\newcommand{\GL}{\text{GL}}
\newcommand{\Sym}{\text{Sym}}
\newcommand{\Spec}{\text{Spec}}
\newcommand\arr{\ifinner\to\else\longrightarrow\fi}

\newcommand{\pr}{\text{pr}}
\newcommand{\ov}{\overline}
\newcommand{\im}{\text{Im}}

\newcommand{\wt}{\widetilde}

\newcommand{\reducible}{\textrm{red}}
\newcommand{\irr}{\text{irr}}
\newcommand{\id}{\text{id}}
\newcommand{\cotd}{\Omega_{\PP(W_1)}(1)}
\newcommand{\sym}{\operatorname{S}}

\newcommand\PP{\mathbb{P}}
\newcommand\cO{\mathcal{O}} 

\newcommand{\ptwo}{\PP^{2}}
\newcommand{\ptwod}{\check{\PP}^{2}}
\newcommand{\hilb}{\operatorname{Hilb}}
\newcommand{\stwo}{\sym^{2}\rT_{\PP(W_1)}(-2)}
\newcommand\rT{\mathrm{T}}
\renewcommand\H{\operatorname{H}}

\newcommand\ZZ{\mathbb{Z}}

\newcommand{\eps}{\epsilon}
\newcommand{\Gr}{\operatorname{Gr}}

\input xy
\xyoption{all}

\title[The integral Chow ring of $\Mcal_3\smallsetminus\Hcal_3$]{The integral Chow Ring of the stack of smooth non-hyperelliptic curves of genus three}
\date{\today}
\author{Andrea Di Lorenzo}
\author{Damiano Fulghesu}
\author{Angelo Vistoli}
\address[Andrea Di Lorenzo]{Aarhus University, Ny Munkegade 118, bldg. 1530, DK-8000 Aarhus C, Denmark}\email{andrea.dilorenzo@math.au.dk}
\address[Damiano Fulghesu]{Department of Mathematics, Minnesota State University, 1104 7th Ave South, Moorhead, MN 56563, U.S.A.}\email{fulghesu@mnstate.edu}
\address[Angelo Vistoli]{Scuola Normale Superiore, Piazza dei Cavalieri 7, 56126 Pisa, Italy}\email{angelo.vistoli@sns.it}
\thanks{The second author has been partially supported by Scuola Normale Superiore and by Simons Foundation grant \#360311. The third author has been partially supported by research funds from the Scuola Normale Superiore}

\begin{document}

\begin{abstract}
We compute the integral Chow ring of the stack of smooth, non-hyperelliptic curves of genus $3$. We obtain this result by computing the integral Chow ring of the stack of smooth plane quartics, by means of equivariant intersection theory.
\end{abstract}

\maketitle

\section*{Introduction}

Moduli stacks of curves play a prominent role in algebraic geometry, and the Chow groups of these stacks have been studied extensively. Their investigation was started by Mumford in his groundbreaking work \cite{Mum}, where in particular it is shown that for $g\geq 2$ the \emph{rational} Chow groups of $\Mcal_g$ and $\overline{\Mcal}_{g}$ admit an intersection product.

Subsequent research ultimately led to the explicit computation of the rational Chow rings $A^*(\Mcal_g)_{\mathbb{Q}}$ for $2\leq g\leq 6$ (\cite{m4, Iza, PV}), while $A^*(\overline{\Mcal}_g)_{\mathbb{Q}}$ has been computed for $g = 2$ in \cite{Mum} and $g = 3$ in \cite{m3bar}.

Thanks to the work of Totaro, Edidin and Graham \cite{EG} we also have a definition of \emph{integral} Chow rings $A^*(\Mcal_g)$  and $A^*(\overline{\Mcal}_g)$ for the moduli stack of smooth curves of genus $g$; after tensoring these with $\mathbb Q$ we get Mumford's rings.

However, these integral Chow rings turn out to be much harder to compute than their rational version; for example, $A^*(\Mcal_2)_{\mathbb Q} = \mathbb Q$, while the integral structure, while known, is complicated. So far, only $A^*(\Mcal_2)$ (see \cite{Vis98}) for $g=2$ and $A^*(\overline{\Mcal}_2)$ (see \cite{la19}) have been computed. In a slightly different direction, the integral Chow rings of the stack $\Hcal_{g}$ of hyperelliptic curves of genus $g$ have been computed (\cite{EFh, DL19}).

All of these calculations are based on the fact the curves involved are double covers of a rational curve. To go beyond the hyperelliptic case, the first case to look at would, of course, be $A^*(\Mcal_3)$; however, this seems quite hard. In this paper we give a presentation for the Chow ring $A^*(\Mcal_3 \smallsetminus\Hcal_{3})$ of the stack of non-hyperelliptic curves of genus $3$. 

Our main result is the following:
\begin{maintheorem}\label{mt1}
	Assume that the base field has characteristic $\neq 2$, $3$ and let $\Mcal_3\smallsetminus\Hcal_3$ be the stack of smooth, non-hyperelliptic curves of genus $3$. Then we have:
	\[ A^*(\Mcal_3\smallsetminus \Hcal_3)\simeq \ZZ[\lambda_1,\lambda_2,\lambda_3]/(9 \lambda_1, 6(\lambda_1^2+4\lambda_2), \lambda_1\lambda_2-28\lambda_3, \lambda_1^3+28\lambda_3) \]
	where $\lambda_i$ is the degree $i$ Chern class of the Hodge vector bundle. 
\end{maintheorem}

It is not clear how to patch $A^*(\Hcal_3)$ and $A^*(\Mcal_3 \smallsetminus\Hcal_{3})$ together to get a presentation for $A^*(\Mcal_3)$.

Main Theorem \ref{mt1} above follows from the calculation of Chow ring of the stack $\Xcal_{4}$ of plane quartics, up to linear change of coordinates. In general, the study of the quotients $\Xcal_d:=[(\PP(W_d)\smallsetminus Z_d)/\GL_3]$, where $\PP(W_d)$ is the projective space of plane curves of degree $d$ and $Z_d$ is the divisor of singular curves, was started by the second and the third authors in \cite{Fu-Vi}, where they compute $A^*(\Xcal_3)$.

Here we prove the following:

\begin{maintheorem}\label{mt2}
	Assume that the base field has characteristic $\neq 2$, $3$. Then we have:
	\[
	A^* (\Xcal_4) \simeq \ZZ[c_1,c_2,c_3,h_4]/ \left( \alpha_1,\alpha_2,\alpha_3, \delta_{\{ 1,3 \}} \right)
	\]
	where
	\begin{eqnarray*}
		\alpha_1 &=& 27 h_4 - 36 c_1,\\ 
		\alpha_2 &=& 9 h_4^2 - 6 c_1 h_4 - 24 c_2,\\
		\alpha_3 &=& h_{4}^3 - c_1 h_{4}^2 + c_2 h_{4} - 28 c_3,\\
		\delta_{\{ 1,3 \}} &=& 55 h_4^3 - 220 c_1 h_4^2 + (280 c_1^2 + 40 c_2) h_4 + (224 c_3- 96 c_1^3 - 128 c_1 c_2).
	\end{eqnarray*}
\end{maintheorem}

The elements $c_i$ appearing above are the Chern classes of the universal vector bundle of rank $3$ over $B\GL_3$ pulled back along the morphism $\Xcal_4\arr B\GL_3$, and $h_4$ is the restriction of the hyperplane section of $\PP(W_4)$, where $W_{4}$ is the space of quartic forms in three variables. 

As to the relations, $\alpha_{1}$, $\alpha_{2}$ and $\alpha_{3}$ are immediately obtained from the general setup of \cite{Fu-Vi}, while $\delta_{\{1,3\}}$ is the fundamental class of the invariant closed subscheme $\ov{Z}_{\{1,3\}}\subset\PP(W_4)$ of quartics that are union of a cubic and a line.

Most of this paper is devoted to prove that the aforementioned elements generate the whole ideal of relations. This is much harder than in the case of $\Xcal_{3}$; it uses all the ideas in \cite{Fu-Vi}, and also requires the development of some new theory. The main difference is that while a cubic with more than one singular point is reducible, irreducible quartics can have up to three singular points. In particular, to get a handle on cycles supported on the locus of irreducible curves with three singular points we need to do some non-trivial work on the Hilbert scheme of length~$3$ subschemes of $\PP^{2}$, which might have some independent interest.
 
\subsection*{Description of contents}

In Section \ref{stack.smooth.plane.curves} we recall the definition of the stack $\Xcal_d$ and its presentation as a quotient stack. We also introduce an equivariant stratification of the closed invariant subscheme $Z_d\subset\PP(W_d)$ of singular plane curves of degree $d$ whose strata $\ov{Z}_\mu^m$ parametrize singular plane curves having certain prescribed irreducible components (see Definition \ref{defi.Zstrat}). We explicitly describe this equivariant stratification in the case $d=4$ by means of a digraph.

In Section \ref{sec.EquivariantIntersectionTheory} we recall some standard facts of equivariant intersection theory, with a specific focus on $\GL_n$-spaces, which are the ones we will have to deal with in the remainder of the paper.

In Section \ref{Xd} we take the first step towards the determination of $A^*(\Xcal_{4})$. In particular, we prove the fundamental Lemma \ref{basic.principle}, which enables us to reduce the computation of the image of $A^{\GL_3}_*(Z_4)\arr A^{\GL_3}_*(\PP(W_4))$ to the computation of the images of $A^{\GL_3}_*(\ov{Z}_\mu^m)\arr A^{\GL_3}_*(\PP(W_4))$, up to cycles coming from the smaller strata. We also write down explicitly the generators of the ideal $I_{\wt{Z}}=(\alpha_1,\alpha_2,\alpha_3)$ (Lemma \ref{alpha.classes.independent}).

In Section \ref{sec.reducible} we investigate the closed subscheme $Z_\reducible\subset\PP(W_4)$ of reducible quartics: we stratify $Z_\reducible$ and we study each stratum separately. We show in Theorem \ref{main.result.reducible.quartics} that there is basically only one relation coming from $Z_\reducible$ that it is not already contained in $I_{\wt{Z}}$, namely the fundamental class $\delta_{\{1,3\}}$.

In Section \ref{sec.irr} we focus on the locally closed stratum $Z^\irr$ of irreducible singular quartics, which we further stratify by number of nodes. In  Subsection \ref{subsec.twosingpoints} we show that there are no new relations coming from the stratum of irreducible singular curves with only two nodes, using an argument that involves the Hilbert scheme of $0$-dimensional closed subscheme of length $2$ in $\PP^{2}$. In Subsection \ref{subsec.threesingularpoints} we show that there are no additional relations coming from the stratum of irreducible singular curves with three nodes.

In Section \ref{sec.modifications} we prove two results on the Hilbert scheme of $0$-dimensional closed subschemes of $\PP^2$ of length $3$, which we needed in the previous Section in order to perform some computations.

In the last Section we put everything together and we compute $A^*(\Xcal_4)$ (Theorem \ref{main.result.quartics}). Finally, we deduce from this an explicit presentation for $A^*(\Mcal_3\smallsetminus\Hcal_3)$ (Theorem \ref{thm.AM3}).

{\bf Acknowledgements.} The authors would like to thank the referee for carefully reading the paper and suggesting significant improvements.

\section{Stack of smooth plane curves of fixed degree.}\label{stack.smooth.plane.curves}

\subsection{The stack $\Xcal_d$}
Let $k$ be a field. We define the moduli stack $\Xcal_d$ as the stack whose objects are diagrams
\[
\xymatrix{
C \ar@{}[r]|-*[@]{\subset} & \Pro(F) \ar[d] \\
& S 
}
\] 
where $F \to S$ is a vector bundle of rank 3 over a $k$-scheme $S$, $\Pro(F) \to S$ is its projectivization, and $C \subset \Pro(F)$ is a Cartier divisor whose restriction to every fiber is a smooth curve of degree $d$. Here, as everywhere else, we follow \cite{Ful} and use the classical convention for the projectivization of a vector bundle, so our $\mathbb{P}(F)$ would be denoted by $\mathbb{P}(F^{\vee})$ in Grothendieck's convention.

We can give $\Xcal_d$ a structure of (global) quotient stack as follows. Let $W_d:=\Sym^d(E)$ where $E$ is the standard representation of $\GL_3$. Let $X:=[X_0,X_1,X_2]$ be a coordinate system of $\Pro(E)$. We can write
\[
\Xcal_d=\left[ (\Pro(W_d)\smallsetminus Z ) / \GL_3 \right]
\]

where $Z \subset \Pro(W_d)$ is the discriminant locus and the action of $\GL_3$ is given by
\begin{eqnarray*}
(A \cdot [f]) (X) \mapsto f \left( A^{-1} \cdot X \right).
\end{eqnarray*}

In \cite{Fu-Vi} the two authors give the more general definition of the stack $\Xcal_{n,d}$ of hypersurfaces of degree $d$ embedded in $\Pro^{n-1}$. The main goal of this paper is to determine $A^*(\Xcal_4)$ so, in order to simplify the notation, we focus on plane curves for general results.

\subsection{The degeneracy locus}
Let $\Delta_d$ be the degeneracy locus of singular $d$-forms in $W_d$. We define $Z_d:=\Pro(\Delta_d) \subset \Pro(W_d)$ and consider the universal curve $C_d \subset \Pro(W_d)\times \Pro(E)$. We have the following equivariant projections
$$
\xymatrix{
C_d \subset \Pro(W_d)\times \Pro(E) \ar[d]_{\pi_1} \ar[r]^{\qquad \quad \pr_1} &\Pro(E)\\
\Pro(W_d) &
}
$$
The hypersurface $C_d$ is given by the bihomogeneous equation of bidegree $(1,d)$
\begin{equation*}
F(X):= \sum_{v \in \N^3(d)} a_{v} X^{v}
\end{equation*}
where $\N^3(d):=\left \{ v \in \N^3 \left | \; |v|=d  \right. \right\}$. 

\begin{rmk}\label{rmk:Ztilde}
	In $\Pro(W_d)\times \Pro(E)$ we can also consider the subvariety $\widetilde{Z}_d$ given by equations \[ \displaystyle \left \{\frac{\partial F(X)}{\partial X_i}=0 \right\}_{i=0,1,2}.\] Notice that the restriction morphism $\widetilde{Z}_d \xrightarrow{\pi_1} Z_d$ is generically 1:1, since the generic singular curve has exactly one nodal point. 
\end{rmk}

\subsection{Stratification of $Z_d$}
In the following, we define the natural factor-wise stratification for $Z_d$. We will use such stratification to determine generators for the ideal $i_* \left( A^{\GL_3}_*(Z_d) \right)$, where $i$ is the closed embedding $i: Z_d \to \Pro(W_d)$ (see Lemma \ref{basic.principle}).

Let us start by giving a general definition for equivariant stratification.

\begin{defi}\label{definition.stratification}
Let $X$ be a $G$-space. An {\it equivariant stratification of $X$} is a finite family $\left\{ Z_{\tau}\right\}_{\tau \in J}$ of locally closed, pairwise disjoint, and equivariant subspaces of $X$, such that $\displaystyle \bigcup_{\tau \in J} Z_{\tau} = X$ and
\[
\ov{Z}_{\tau} \smallsetminus Z_{\tau} = \bigcup_{\tau' \in J(\tau)} Z_{\tau'}.
\]
for a suitable subset $J(\tau) \subset J$.

We will also equip an equivariant stratification with a partial ordering such that $Z_{\tau'} < Z_{\tau}$ whenever $Z_{\tau'} \subset \ov{Z}_{\tau} \smallsetminus Z_{\tau}$. Moreover, we also define $\partial Z_{\tau}:=\ov{Z}_{\tau} \smallsetminus Z_{\tau}$.
\end{defi}

Let $d$ be a positive integer. We define a {\it weighted partition} of $d$ the datum of a multiset of positive integers $\mu=\{ k_1, k_2, \dots, k_s\}$ and a multiplicity vector $m=(m_1, m_2, \dots, m_s) \in \left( \Z^+ \right)^s$ such that
\[
\sum_{j=1}^{s} k_j m_j =d.
\]
We will always assume $k_1 \leq k_2 \leq \dots \leq k_s$. We will denote the set of all weighted partitions of $d$  as $\Wcal \Pcal_{d}$.

\begin{rmk}
Now, we clarify some details about the notation we use for weighted partitions. Let $\mu=\{ k_1, k_2, \dots, k_s\}$ be a multiset such that $k_1 \leq k_2 \leq \dots \leq k_s$ and let $m=(m_1, m_2, \dots, m_s)$ be a multiplicity vector for $\mu$. We consider the pair $(\mu,m)$, an element of $\Wcal \Pcal_{d}$, where $d= \sum_{j=1}^{s} k_j m_j$. Now, assume that for some positive integers $j$ and $q$ we have $k_{j} = k_{j+1}= \dots =k_{j+q}$. Any rearranging of elements $m_{j}, m_{j+1}, \dots, m_{j+q}$ will represent the same element $(\mu, m)$. As a convention, we will order the weights in such a way that $m_{j} \geq m_{j+1} \geq \dots \geq m_{j+q} $.

As an example, the weighted partitions 
\[ (\{ 1,1,1,2,2,3\}, (1,2,3,2,3,1)),\quad (\{ 1,1,1,2,2,3\}, (1,3,2,3,2,1))\]
are actually the same, which we will write as $(\{ 1,1,1,2,2,3\}, (3,2,1,3, 2,1))$.
\end{rmk}

\begin{defi}\label{defi.Zstrat}
Let $(\mu,m) \in \Wcal \Pcal_{d}$ be a weighted partition of $d$ where $\mu=\{ k_1, k_2, \dots, k_s\}$ and $m=(m_1,m_2, \dots, m_s)$. Then we define $Z_{\mu}^m$ as the set of curves of degree $d$ which are union of distinct irreducible components of degree $k_j$ and multiplicity $m_j$, as $j$ runs from 1 to $s$.
\end{defi}

For example, $Z_{\{ 1,1,1,2,2,3\}}^{(3,2,1,3, 2,1)}$ is the set of plane curves of degree 19 which are union of a triple line, a double line, a reduced line, a triple quadric, a double quadric, and a reduced cubic. Moreover, the lines, quadrics, and the cubic are all different and irreducible.

\begin{defi}
Let $(\mu, m) \in \Wcal \Pcal_{d}$. We define the map
\begin{eqnarray*}
W_{\mu} := \prod_{j=1}^{s} \Pro(W_{k_j}) & \xrightarrow{\pi_{\mu}^{m}} & \Pro(W_d)\\
\left( [f_1], [f_2], \dots, [f_s] \right) & \mapsto &  \left[ f_1^{m_1}f_2^{m_2} \dots f_s^{m_s} \right].
\end{eqnarray*}
\end{defi}

Notice that the map $\pi_{\mu}^{m}$ is proper, finite, and generically unramified.

\begin{lemma}\label{lemma.closure.strata}
Let $(\mu, m) \in \Wcal \Pcal_{d}$. We have the identity:
\[ \pi_{\mu}^{m} \left( W_{\mu} \right) = \ov{Z}_{\mu}^{m}. \]
\end{lemma}

\begin{proof}
By definition and by unique factorization, we have the inclusion $ Z_{\mu}^{m} \subset \pi_{\mu}^{m} \left( W_{\mu} \right)$. Since $W_{\mu}$ is irreducible, we have that $\pi_{\mu}^{m} \left( W_{\mu} \right)$ is irreducible. Moreover, because the map $\pi_{\mu}^{m}$ is proper, $\pi_{\mu}^{m} \left( W_{\mu} \right)$ is closed in $\Pro(W_d)$. Therefore, we have the closed embedding $\ov{Z}_{\mu}^{m} \subseteq \pi_{\mu}^{m} \left( W_{\mu} \right)$. Now, we notice that $\left( \pi_{\mu}^{m} \right)^{-1}\left( Z_{\mu}^{m} \right)$ is open in $W_{\mu}$, consequently, $Z_{\mu}^{m}$ is open in $\pi_{\mu}^{m} \left( W_{\mu} \right)$. This implies that $\dim \left( Z_{\mu}^{m} \right) = \dim \left( \pi_{\mu}^{m} \left( W_{\mu} \right) \right)$. Thus (see, for example, \cite[Theorem I.6.1]{Shaf}):
\[ \pi_{\mu}^{m} \left( W_{\mu} \right) = \ov{Z}_{\mu}^{m}.\]
\end{proof}

\begin{rmk}\label{factor-wise.stratification}
With the help of Lemma \ref{lemma.closure.strata} and by using unique factorization, it is straightforward to prove that $\left\{ Z_{\mu}^{m} \right\}_{(\mu,m) \in \Wcal \Pcal_d}$ is an equivariant stratification of $\Pro(W_d)$.

We also notice that by substituting the set $Z_{\{ d\}}$ of the irreducible curves of degree $d$ with the set $Z_{\irr}$ of irreducible singular curves, we get an equivariant stratification of $Z_d$. We call such stratification the {\it factor-wise stratification} of the degeneracy locus.

We have exactly the same kind of stratification for general hypersurfaces.
\end{rmk}

Since the factor-wise stratification of the singular locus $Z_d$ is equipped with a natural partial ordering, we can represent such stratification with a digraph. 

In the remainder of the paper we will be mostly concerned with the case $d=4$. We can represent the factor-wise stratification of $Z_4 \subset \Pro(W_4)$ with the following digraph. On the left we write the codimension of the corresponding strata.
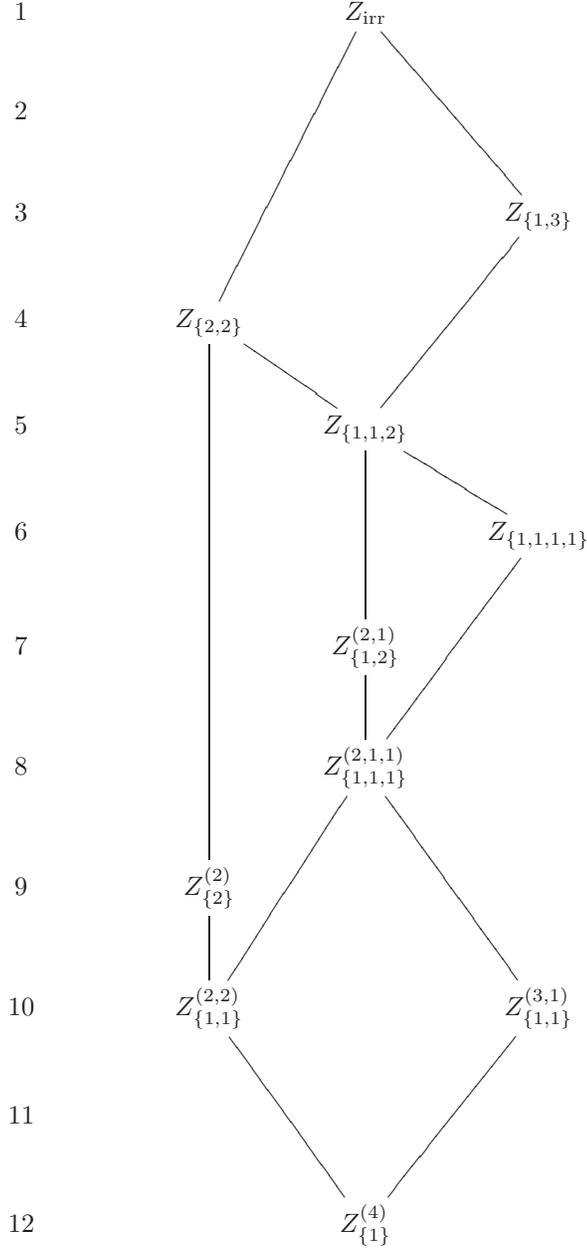
\begin{equation}\label{eq:strat}
\xymatrix{
	\text{Codimension} & & & \\
	1 & & Z_{\irr} \ar@{-}[ddr] \ar@{-}[dddl]& \\
	2 & & &\\
	3 & & & Z_{\{ 1,3\}} \ar@{-}[ddl]\\
	4 & Z_{\{ 2,2 \}} \ar@{-}[dr] \ar@{-}[ddddd]& &\\
	5 & & Z_{\{ 1,1,2 \}} \ar@{-}[dr] \ar@{-}[dd]&\\
	6 & & & Z_{\{ 1,1,1,1\}} \ar@{-}[ddl]\\
	7 & & Z^{(2,1)}_{\{ 1, 2\}} \ar@{-}[d]& \\
	8 & & Z^{(2,1,1)}_{\{ 1, 1, 1\}}  \ar@{-}[ddr]  \ar@{-}[ddl] & \\
	9 & Z^{(2)}_{\{ 2 \}} \ar@{-}[d] & & \\
	10 & Z^{(2,2)}_{\{ 1,1 \}}  \ar@{-}[ddr] & & Z^{(3,1)}_{\{ 1,1\}}  \ar@{-}[ddl] \\
	11 & & & \\
	12 & & Z^{(4)}_{\{ 1 \}} &
}
\end{equation}
Notice that, by definition, we have $Z_{\reducible} = \ov{Z}_{\irr} \smallsetminus Z_{\irr} = \ov{Z}_{\{ 1,3 \}} \cup \ov{Z}_{\{ 2,2\}}$.
Moreover, we set $Z_{\rm binod}$ as the stratum of irreducible singular curves with two nodes, and $Z_{\rm trinod}$ as the stratum of irreducible singular curves with three nodes. The closure of this strata will be denoted $\ov{Z}_{\rm binod}$ and $\ov{Z}_{\rm trinod}$.

\section{Equivariant intersection theory of $\GL_3$-spaces}\label{sec.EquivariantIntersectionTheory}

In this Section we collect some facts on the equivariant intersection rings of some basic $\GL_3$-spaces over a base field $k$ by following the theory developed in \cite{EG}. All these results can also be found in \cite{EF}, \cite{EFh}, and \cite{Fu-Vi}.

Since in this paper we only focus on plane curves, in order to simplify the notation, we consider the case of $\GL_3$-spaces, but all the statements can be easily extended so to hold for $\GL_n$-spaces in general. 

Define
\[ A^{*}_{\GL_3}:=A^{*}_{\GL_3}(\Spec(k))\simeq A^*(B\GL_3) \]
where $B\GL_3$ is the classifying stack of $\GL_{3}$-torsors, which can also be regarded as the classifying stack of rank $3$ vector bundles. It is well known that $B\GL_3\simeq \left[\Spec(k)/\GL_3 \right]$.

The intersection ring $A^{*}_{\GL_3}(\Pro(F))$ is naturally equipped with a $A^*_{\GL_3}$-module structure, hence we start by recalling an explicit presentation of $A^*_{\GL_3}$.

Let $E$ be the standard representation of $\GL_3$, so that the quotient stack $\Ecal:=[E/\GL_3]$ is the universal rank $3$ vector bundle over $B\GL_3$.

Let $T\subset \GL_3$ be the maximal subtorus of diagonal matrices, so that we have $T\simeq\mathbb{G}_m^{\oplus 3}$, and define $L_i$ as the $1$-dimensional representation of $T$ induced by the projection on the $i^{\rm th}$-factor, for $i=1$, $2$ or $3$. The quotient stacks $\Lcal_i:=[L_i/T]$ are then line bundles over the classifying stack $BT$.

In particular, we have the following cartesian diagram:
\begin{equation}
\xymatrix{ \Lcal_1\oplus\Lcal_2\oplus\Lcal_3 \ar[r] \ar[d] & \Ecal \ar[d] \\
			BT \ar[r]^{j} & B\GL_3 }
\end{equation}
where the bottom arrow is induced by the injective homomorphism $T\hookrightarrow\GL_3$.

\begin{propos}[{\cite[Prop. 6]{EG}}]
	Define $\ell_i:=c_1(\Lcal_i)$ in the integral Chow ring of $BT$, and set $c_i:=c_i(\Ecal)$ in the integral Chow ring of $B\GL_3$.
	Then we have:
		\[ A^*(BT)=A^*_T\simeq\ZZ[\ell_1,\ell_2,\ell_3],\quad A^*(B\GL_{3})=A^*_{\GL_3}\simeq\ZZ[c_1,c_2,c_3].\]
	Moreover, the pullback morphism $j^*:A^*_{GL_3}\rightarrow A^*_T$ is injective, and its image corresponds to $(A^*_T)^{S_3}$, where $S_3$ acts on $A^*_T$ by permuting the generators. In particular we have:
		\[j^*c_1=\ell_1+\ell_2+\ell_3,\quad j^*c_2=\ell_1\ell_2+\ell_2\ell_3+\ell_1\ell_3,\quad j^*c_3=\ell_1\ell_2\ell_3. \]
\end{propos}

We can think of $\ell_1$, $\ell_2$ and $\ell_3$ as the Chern roots of $\Ecal$. For notational convenience, we also introduce the classes
\begin{equation}\label{eq.l}
l_1:=c_1(\Lcal_1^\vee),\quad l_2:=c_1(\Lcal_2^\vee),\quad l_3:=c_1(\Lcal_3^\vee).
\end{equation}
In particular, we have $c_1 =-(l_1 + l_3 + l_n)$,
$c_2 = l_1l_2 + l_1l_3 + l_2l_3$, and
$c_3 = -l_1l_2l_3$.

The following Lemma will allow us to perform computations on $T$-equivariant Chow groups rather than the $\GL_3$-equivariant ones. 

\begin{lemma}[{\cite[Prop.\ 6]{EG}, \cite[Lemma 2.1]{Fu-Vi}}]
	The following are true:
	\begin{enumerate}
		\item Let $G$ be a special group and let $T$ be the maximal subtorus of $G$ with Weyl group $W$. Then if $X$ is a smooth $G$-scheme, the pullback along the morphism $[X/T]\to [X/G]$ induces an injective homomorphism 
		\[ A_G^*(X)\hookrightarrow A_T^*(X),\] 
		whose image coincides with $A^*_T(X)^W$.
		\item 	For every ideal $I\subset A^*_G(X)$ we have:
		\[ IA^*_T(X)\cap A_G^*(X) =  I.\]
		\item Let $f:X\to Y$ be a $G$-equivariant proper morphism between smooth $G$-schemes, and suppose that the image of $f_*:A_T^*(Y)\to A_T^*(X)$ is generated as $A^*_T$-module by $W$-invariant classes $\gamma_1,\ldots,\gamma_n$. Then the image of $f^*:A_G^*(Y)\to A^*_G(X)$ is also generated by $\gamma_1,\ldots,\gamma_n$.
	\end{enumerate}
\end{lemma}

The main advantage of working with $T$-equivariant Chow rings is given by the so-called \emph{explicit localization formula}, which we state below.
\begin{theorem}[{\cite[Theorem 2]{EGL}}]\label{thm.loc}
	Define the $A^*_T$-module $ \Qcal:= ((A^*_T)^+)^{-1}A^*_T$, 
	where $(A^*_T)^+$ is the multiplicative system of homogeneous elements of $A^*_T$ of positive degree.
	
	Let $X$ be a smooth $T$-variety and consider the locus $F$ of fixed points for the action of $T$. Let $F=\cup F_j$ be the decomposition of $F$ into irreducible components. For every $\gamma$ in $A^*_T(X)\otimes\Qcal$, we have the identity
	\[ \gamma = \sum_{j} \frac{i_{F_j}^*(\gamma)}{c_{top}^T(N_{F_{j}}X)} \]
	where $i_{F_j}$ is the inclusion of $F_j$ in $X$ and $N_{F_j}X$ is the normal bundle of $F_j$ in $X$.
\end{theorem}

\begin{rmk}\label{rmk.loc}
In particular, when $A^*_T(X)$ is torsion-free as $A^*_T$-module, the localization homomorphism $A^*_T(X)\to A^*_T(X)\otimes\Qcal$ is injective, and for every $T$-equivariant morphism $f:Y\to X$ of smooth $T$-varieties, we have a commutative diagram
\[ \xymatrix{A^*_T(Y) \ar[r]^{f_*} \ar[d] & A^*_T(X) \ar[d] \\ A^*_T(Y)\otimes\Qcal \ar[r] & A^*_T(X)\otimes\Qcal } \] 
If $F_j$ is a point for every $j$, we have:
\begin{equation}\label{eq:loc} f_*\gamma = \sum_j \frac{[f(F_j)]\cdot p_j(\gamma)}{c_{top}^T(TY_{F_j})} \end{equation}
where  $\gamma \in A^*_T(Y)$ and $p_j(\gamma)=i_{F_j}^*\gamma$ in $A^*_T(F_j)=A^*_T$. We will mostly use the localization formula in the form of (\ref{eq:loc}).
\end{rmk}

Let us recall the equivariant version of the projective bundle formula.
\begin{propos}[{\cite[2.3]{Fu-Vi}}]\label{prop:proj bundle formula}
	Let $F$ be a $G$--representation of dimension $n$. Then we have:
	\[ A_{G}^*(\PP(F))\simeq A_{G}^*[h]/(h^n+c_1^{G}(F)h^{n-1}+c_2^{G}(F)h^{n-2}+\ldots +c_n^{G}(F)). \]
	Let $F_1,\ldots,F_s$ be $s$ representations. Then we have:
	\[ A^*_{G}(\PP(F_1)\times \cdots\times \PP(F_s))\simeq A^*_{G}(\PP(F_1))\otimes_{A^*_{G}}\cdots\otimes_{A^*_{G}}A_G(\PP(F_s)). \]
\end{propos}

We conclude by mentioning another useful feature of working with $T$-equivariant Chow groups, namely the fact that we can easily compute the cycle classes of complete intersections via the Lemma below.

\begin{lemma}[{\cite[Lemma 2.4]{EF}}]
	Let $W$ be a $T$-representation of finite rank, and let $H\subset\PP(W)$ be a $T$-invariant hypersurface. Suppose that $H$ is defined by a homogeneous equation $F(X)=0$ of degree $d$, and that
	\[ F(\tau^{-1}\cdot X) = \chi(\tau)F(X) \]
	for some $T$-character $\chi$ and for every $\tau\in T$. Then:
	\[ [H]_T = dh + c_1(\chi) \in A^*_T(\PP(W)) \]
	where $c_1(\chi)$ is the $T$-equivariant first Chern class of the $T$-representation defined by $\chi$.
\end{lemma}

Given $Y\subset\PP(W)$ a complete intersection of $s$ invariant hypersurfaces $H_1,\dots,H_s$, by the Lemma above we easily deduce
\[ [Y]_T= (d_1h + c_1(\chi_1))\cdot \cdots \cdot (d_sh + c_1(\chi_s)). \]

\section{Intersection theory on $\Xcal_4$}\label{Xd}

The aim of this Section is to take the first steps towards the computation of $A^*(\Xcal_4)$.
From now on we assume that the base field $k$ has characteristic $\neq 2,3$. 

\begin{rmk}
If the characteristic of the base field is equal to $2$ several things can go wrong. For instance, the subscheme $\widetilde{Z}_4\subset \PP(W_4)\times\PP^2$ is no more generically $1:1$ over $Z_4$, the reason being that the partial derivatives $F_{X_i}$ are not linear independent, because $\sum p_iF_{X_{i}}(p_1,p_2,p_3)=\deg(F)\cdot F=0$. The same formula shows that in this case $\widetilde{Z}_4$ does not coincide with the singular locus of the universal quartic curve.

Moreover, in the proofs of some results (e.g. Lemma \ref{lemma.2Z1111}) we rely on computations performed in \cite{Fu-Vi} for cubic curves, and for these computations to be true we need to assume that the characteristic of the base field is not $3$.  
\end{rmk}

We start by applying Proposition \ref{prop:proj bundle formula} to the $\GL_3$-representation $W_4$.
\begin{corol}
	We have:
	\[ A^*_{\GL_3}(\PP(W_4))\simeq \ZZ[c_1,c_2,c_3,h_4]/(p_4(h_4)), \]
	where $p_4(h_4)$ is a homogeneous polynomial of degree $15$, monic with respect to the variable $h_4$.
\end{corol}

Recall that $\Xcal_4\simeq [(\PP(W_4)\smallsetminus Z_4)/\GL_{3}]$, where $Z_4=\PP(\Delta)$ is the divisor of singular quartics. Then the excision exact sequence for equivariant Chow groups reads as follows:
\[ A^{\GL_3}_*(Z_4) \xrightarrow{i_{*}} A^{\GL_3}_*(\PP(W_4))\simeq \ZZ[c_1,c_2,c_3,h_4]/(p_4(h_4)) \arr A_*(\Xcal_4)\arr 0 \]
We will leverage the next Lemma in order to investigate the image of $i_*$.
\begin{lemma}\label{basic.principle}
Let $X$ be a quasi-projective and equidimensional space with a linearized $G$-action. Let $\left \{ Z_{\tau} \right \}_{\tau \in J}$ be an equivariant stratification of $X$ (see Definition \ref{definition.stratification}). Fix $\tau \in J$ and let $N$ be an $A^*_G$-submodule of $A^G_*(X)$. Assume that for every $\tau' < \tau$ we have $i_{*} \left( A^G_*(\ov{Z}_{\tau'})\right) \subset N$ (where $i$ is the closed embedding $i: \ov{Z}_{\tau'} \to X$) and assume that we have the inclusion $i_{*} \left( A^G_* \left( Z_{\tau} \right) \right) \subset j^*(N)$ (where $i$ is the closed embedding $i: Z_{\tau} \to X \smallsetminus \partial Z_{\tau}$ and $j$ is the open immersion $j: X \smallsetminus \partial Z_{\tau} \to X$), then $i_* \left( A^G_* (\ov{Z}_{\tau})\right) \subset N$.
\end{lemma}

\begin{proof}
If $\tau$ is minimal, the statement is trivial. If $\tau$ is not minimal, the map
\[
p: \bigsqcup_{\tau' < \tau} Z_{\tau'} \arr \partial Z_{\tau}
\]
is an equivariant envelope, therefore the push-forward 
\[
p_*: A^G_* \left( \bigsqcup_{\tau' < \tau} Z_{\tau'} \right) \arr A^G_* \left( \partial Z_{\tau} \right)
\]
is surjective. Together with the hypothesis, this implies that $i_* \left( A^G_* (\partial Z_{\tau})\right) \subset N$ (where $i$ is the closed embedding $i: \partial Z_{\tau} \to X$).

Now, consider the following commutative diagram of $A^G_*$-modules, where $i$ represents a closed embedding and $j$ represents an open immersion:
\[
\xymatrix{
A^G_* \left( \partial Z_{\tau} \right) \ar[r]^{i_*} \ar[d]^{\id_*} & A^G_*\left( \ov{Z}_{\tau} \right) \ar[r]^{j^*} \ar[d]^{i_*}& A^G_* \left( Z_{\tau} \right) \ar[r] \ar[d]^{i^*}& 0 \\
A^G_* \left( \partial Z_{\tau} \right) \ar[r]^{i_*} & A^G_*\left( X \right) \ar[r]^{j^*} & A^G_* \left( X \smallsetminus \partial Z_{\tau} \right) \ar[r] & 0
}
\]

Moreover, notice that, from \cite[Proposition 5]{EG}, the rows are exact. Let $\gamma$ be a class in $ A^G_*\left( \ov{Z}_{\tau} \right)$. By using a standard diagram chasing argument, we can show that $i_*(\gamma) \in N$.
\end{proof}

\begin{defi}\label{defi:IZtilde}
	Consider the $\GL_3$-variety $\PP(W_4)$.
	\begin{enumerate}
		\item We define the ideal $I_{\wt{Z}}\subset A^*_{\GL_{3}}(\PP(W_4))$ as the image of
		\[ i_*\pi_{1*}:A^{\GL_{3}}_*(\wt{Z}_4)\arr A_*^{\GL_3}(\PP(W_4)),\]
		where $\wt{Z}_4\subset\PP(W_4)\times\PP(E)$ is the singular locus of the universal plane quartic, as defined in Remark \ref{rmk:Ztilde}.\\
		\item We define $\delta_{\mu}^m$ as the equivariant cycle class $[\ov{Z}_{\mu}^m]_{\GL_3}$ in $A^*_{\GL_3}(\PP(W_d))$.
	\end{enumerate}
\end{defi}

Lemma \ref{basic.principle} will be applied to the stratification (\ref{eq:strat}), where the submodule $N$ that we take into account is the ideal $(I_{\wt{Z}},\delta_{\{1,3\}})$. Our final result will be that the image of $i_*$ is contained in $(I_{\wt{Z}},\delta_{\{1,3\}})$, hence these two ideals must coincide.

We can easily determine the classes $\alpha_1,\alpha_2,\alpha_3$ which generate the ideal $I_{\wt{Z}}$. We apply \cite[Theorem 4.5]{Fu-Vi} and we get:

\begin{eqnarray*}
	\alpha_1 &=& 27 h_4 - 36 c_1,\\ 
	\alpha_2 &=& 9 h_4^2 - 6 c_1 h_4 - 24 c_2,\\
	\alpha_3 &=& h_4^3 - c_1 h_4^2 + c_2 h_4 - 28 c_3.\\
\end{eqnarray*}

\begin{lemma}\label{alpha.classes.independent}
	The classes $\alpha_1, \alpha_2,\alpha_3$ are independent generators for the ideal $I_{\wt{Z}}$. Moreover, the polynomial $p_4(h_4)\in I_{\wt{Z}}$.
\end{lemma}

\begin{proof}
	We already know that the classes $\alpha_i$ generate the ideal $I_{\wt{Z}}$, therefore, it is enough to show that they are independent. If we consider the classes mod 9, we see that $\alpha_2$ is independent from $\alpha_1$. On the other hand, if we consider the classes mod 3, we see that $\alpha_3$ is independent from $\alpha_1$ and $\alpha_2$. Since each $\alpha_i$ is homogeneous of degree $i$, this is enough to prove the first statement.
	
	The second statement is \cite[Prop. 4.8]{Fu-Vi}.
\end{proof}

\section{Reducible quartics}\label{sec.reducible}

Let $Z_\reducible \subset \Pro(W_4)$ be the locus of reducible quartics. The goal of this Section is to prove that $i_* \left( A_*^{\GL_3} (Z_\reducible) \right)$ is contained in the ideal generated by $I_{\wt{Z}}$ and the class $\delta_{\{1,3\}}$.

\subsection{The push-forward of classes from $\ov{Z}_{\{1,1,1,1\}}$ is in the ideal $I_{\wt{Z}}$.}

The main goal of this Subsection is to show the following:
\begin{propos}\label{propos.Z1111}
		The image of $i_*:A_{\GL_3}^*(\ov{Z}_{\{1,1,1,1\}})\arr A_{\GL_3}^*(\PP(W_4))$ is contained in $I_{\wt{Z}}$.
\end{propos}

First we show that the pushforward of the classes coming from the border of $\ov{Z}_{\{1,1,1,1\}}$ are in $I_{\wt{Z}}$.

\begin{lemma}\label{borderT4}
We have the inclusion $i_* \left( A^{\GL_3}_* ( \partial \ov{Z}_{\{1,1,1,1\}}) \right) \subseteq I_{\wt{Z}}$.
\end{lemma}

\begin{proof}
Recall that $\wt{Z}_4$ is the singular locus of the universal quartic $\pi_1:C_4\arr \Pro(W_4)$. By definition $I_{\wt{Z}}=\pi_{1*}A^{\GL_3}_*(\wt{Z}_4)$, hence to prove the Proposition it is enough to show that the restriction of $\wt{Z}_4$ over $\partial \ov{Z}_{ \{1,1,1,1\} }$ is a Chow envelope.

First observe that the restriction of $\wt{Z}_4$ over $\ov{Z}_{\{1\}}^{(4)}$, endowed with the reduced structure, is the projectivization of an equivariant vector bundle, thus a Chow envelope.

The restriction of $\wt{Z}_4$ either over $Z_{\{1,1\}}^{(2,2)}$ or $Z_{\{1,1\}}^{(3,1)}$ has a section given by sending a point below to the unique node of the corresponding plane curve. Henceforth, in both cases we have a Chow envelope.

Finally, the restriction of $\wt{Z}_4$ over $Z_{\{1,1,1\}}^{(2,1,1)}$ also has a section, namely the one that sends a point below to the unique node where the two reduced components of the associated plane curve meet. Putting all together, this implies that $\wt{Z}_4$ is a Chow envelope for $\partial \ov{Z}_{\{1,1,1,1\}}$.
\end{proof}

\begin{lemma}\label{lemma.2Z1111}
We have the inclusion
\[
4 i_* \left( A_*^{\GL_3} \left( \ov{Z}_{\{1,1,1,1\}} \right) \right) \subseteq I_{\wt{Z}}.
\]
\end{lemma}

\begin{proof}

Denote by $T_{d}:=\ov{Z}_{\{1,1,\dots, 1\}} \subseteq \Pro(W_d)$ the closed subvariety consisting of sums of lines (which can appear with multiplicity larger than 1), $i\colon Z_{d} \subseteq \Pro(W_d)$ the discriminant together with its inclusion $i$ in $\Pro(W_d)$, and $\widetilde{Z}_{d} \arr Z_{d}$ the usual resolution. We need to show that if $\gamma \in A_*^{\GL_3}(T_{4})$, then $4i_{*}\gamma \in I_{\widetilde{Z}}$.

There is an obvious morphism $\pi\colon \Pro(W_{1}) \times T_{3} \arr T_{4}$, which is generically $4:1$. By considering the special case of \cite[Lemma 5.12]{Fu-Vi} where $\Gamma$ is the trivial group, we can write the following relation
\[
4(\gamma - \gamma') = \pi_{*}\overline{\gamma}
\]

where $\gamma' \in p_*\left( A^{\GL_3}_*(\partial T_4)\right)$ and $\ov{\gamma} \in A^{\GL_3}_* (\Pro(W_{1}) \times T_{3})$ (here we denote with $p$ the closed embedding of the border $\partial T_4$ in $T_4$). In other words, we can write $4\gamma$ as the sum of (4 times) a class $\gamma'$ coming from the border $\partial T_4$ and the pushforward of a class $\ov{\gamma}$ coming from $\Pro(W_{1}) \times T_{3}$.

Since from Proposition \ref{borderT4}, we know that $i_*\left( A^{\GL_3}_*(\partial T_4) \right) \in I_{\wt{Z}}$, if we show that the pushforward of $\overline{\gamma}$ in $\Pro(W_{4})$ is in $I_{\widetilde{Z}}$, then we know that $4i_{*}\gamma$ in $I_{\widetilde{Z}}$. 

The composite $\Pro(W_{1}) \times T_{3} \arr \Pro(W_{4})$ factors through $\Pro(W_{1}) \times \Pro(W_{3})$. From the result for cubics (see \cite[Corollary 5.15]{Fu-Vi}), we know that the image of $A^{\GL_{3}}_*(T_{3}) \arr A^{\GL_3}_*(\Pro(W_{3}))$ is contained in the image of $A^{\GL_3}_*(\widetilde{Z}_{3}) \arr A^{\GL_3}_*(\Pro(W_{3}))$; it follows that the image of $A^{\GL_3}_*(\Pro(W_{1}) \times T_{3}) \arr A^{\GL_3}_*(\Pro(W_{1}) \times \Pro(W_{3}))$ is contained in the image of $A^{\GL_3}_*(\Pro(W_{1}) \times \widetilde{Z}_{3}) \arr \A(\Pro(W_{1}) \times \Pro(W_{3}))$. But the composite $\Pro(W_{1}) \times \widetilde{Z}_{3} \arr \Pro(W_{1})\times \Pro(W_{3}) \arr \Pro(W_{4})$ factors through $\widetilde{Z}_{4}$, which proves the result.

\end{proof}

In order to prove Proposition \ref{propos.Z1111} we need some additional theory  which will also be useful when discussing the case of irreducible singular curves. For the sake of readability, we preferred to postpone the detailed development of such theory to Subsection \ref{subsec.twosingpoints}. For the moment, we will content ourselves with a brief sketch of those results that are strictly necessary only for the present Subsection.

Let $\hilb^2\PP^2$ be the Hilbert scheme of $0$-dimensional subschemes in $\PP^2$ of length $2$. Consider the invariant set $D\subset\PP(W_4)\times \hilb^2\PP^2$ formed by the pairs $([Q],[S])$ where $S$ is schematically contained in the singular locus of the plane quartic $Q$. We show in Subsection \ref{subsec.twosingpoints} that $D$ is actually a smooth closed subscheme of $\PP(W_4)\times\hilb^2\PP^2$, and the image of the pushforward $p_*:A^{\GL_3}_*(D)\arr A^{\GL_3}_*(\PP(W_4))$ is contained in $I_{\wt{Z}}$ (Proposition \ref{propos.twonodes}).

\begin{lemma} \label{lemma.15Z1111}
	We have the inclusion
	\[
	15 \cdot i_* \left( A_*^{\GL_3} \left( \ov{Z}_{\{1,1,1,1\}} \right) \right) \subseteq I_{\wt{Z}}.
	\]
\end{lemma}
\begin{proof}
	We already know from Lemma \ref{borderT4} that $i_*A^{\GL_3}_*(\partial \ov{Z}_{\{1,1,1,1\}})$ is contained in $I_{\wt{Z}}$. By Lemma \ref{basic.principle} it is enough to prove that $15 \cdot i_* A_*^{\GL_3}(Z_{\{1,1,1,1\}})$ is contained in the image of $I_{\wt{Z}}$ along the quotient morphism \[A_{\GL_3}^*(\PP(W_4))\arr A_{\GL_3}^*(\PP(W_4)\setminus \partial \ov{Z}_{\{1,1,1,1\}}).\]
	
	Observe that $D|_{Z_{\{1,1,1,1\}}}\arr Z_{\{1,1,1,1\}}$ is finite of degree $15$: this implies that for any cycle $\gamma$ in $A^{\GL_3}_*(Z_{\{1,1,1,1\}})$ the element $15\gamma$ is in the image of \[A^{\GL_3}_*(D|_{Z_{\{1,1,1,1\}}})\arr A^{\GL_3}_*(Z_{\{1,1,1,1\}}).\]
	
	Moreover, the morphism 
	\[ D|_{Z_{\{1,1,1,1\}}}\arr Z_{\{1,1,1,1\}} \arr \PP(W_4)\setminus \partial \ov{Z}_{\{1,1,1,1\}}\]
	obviously factors through $D|_{\PP(W_4)\setminus \partial \ov{Z}_{\{1,1,1,1\}}}$. By Corollary \ref{cor.twonodes} the image of \[ A^{\GL_3}_*(D|_{\PP(W_4)\setminus \partial \ov{Z}_{\{1,1,1,1\}}})\arr A^{\GL_3}_*(\PP(W_4)\setminus \partial \ov{Z}_{\{1,1,1,1\}})\] is contained in the image of $I_{\wt{Z}}$. Putting all together, we deduce that $i_*(15\gamma)$ is in the image of $I_{\wt{Z}}$.
\end{proof}
Combining Lemma \ref{lemma.2Z1111} and Lemma \ref{lemma.15Z1111} we immediately deduce Proposition \ref{propos.Z1111}.

\subsection{The push-forward of classes from $\ov{Z}_{\{2,2\}}$ is in the ideal $I_{\wt{Z}}$.}

The main goal of this Subsection is to prove the following:
\begin{propos}\label{propos.Z22}
	The image of $i_*:A_{\GL_3}^*(\ov{Z}_{\{2,2\}})\arr A_{\GL_3}^*(\PP(W_4))$ is contained in $I_{\wt{Z}}$.
\end{propos}
We start with a preliminary result.
\begin{propos}\label{propos.Z2}
	The image of $i_*:A_{\GL_3}^*(\ov{Z}_{\{2\}}^{(2)})\arr A_{\GL_3}^*(\PP(W_4))$ is contained in $I_{\wt{Z}}$.
\end{propos}
\begin{proof}
	We have to compute the image of the pushforward of cycles along the square morphism $\pi:\PP(W_2)\arr\PP(W_4)$: this ideal will be generated by the cycles $\pi_*h_2^i$ for $i=0,1,\dots,5$. We proceed by localization (Theorem \ref{thm.loc} and Remark \ref{rmk.loc}).
	
	The $T$-fixed points in $\PP(W_2)$, where $T$ is the torus of diagonal matrices, correspond to the monomials of degree $2$ in three variables. We adopt the following notation:
	\begin{eqnarray*}
		Q_1=X^2 & Q_4=Y^2 \\
		Q_2=XY & Q_5=YZ \\
		Q_3= XZ & Q_6=Z^2
	\end{eqnarray*}
	Using the generators $l_1$, $l_2$ and $l_3$ introduced in (\ref{eq.l}), let $\chi_i$ indicate the character associated to the monomial $Q_i$, so that we have:
	\begin{eqnarray*}
		c_1(\chi_1)=2 l_1 & c_1(\chi_4)=2 l_2 \\
		c_1(\chi_2)= l_1+l_2 & c_1(\chi_5)= l_2+l_3 \\
		c_1(\chi_3)= l_1+l_3 & c_1(\chi_6)=2 l_3 
	\end{eqnarray*}
	where $c_1(\chi_i)$ means the $T$-equivariant first Chern class of the $T$-representation defined by $\chi_i$. For the sake of simplicity, in what follows we will use the notation $\chi$ to denote both the character and its associated first Chern class.

	Use the coordinates $(a_1:\dots:a_6)$ to indicate the point $[a_1Q_1+\cdots + a_6Q_6]$ in $\PP(W_2)$. The action of $T$ on these coordinates is given by the formula $t\cdot a_i=-\chi_i(t)a_i$. Then we have:	
	\[ T(\PP(W_2))_{[Q_i]}\simeq \Big\langle \dfrac{a_k^\vee}{a_i^\vee} \Big\rangle_{k\neq i} \Longrightarrow \quad c_{top}^T(T\PP(W_2)_{[Q_i]})=\prod_{k\neq i} (\chi_k -\chi_i).  \]
	From this we also deduce that the restriction of $h_2$ to the equivariant Chow ring of the fixed point $[Q_i]$ is equal to $-\chi_i$.
	
	Finally, set:
	\begin{align*}
	F_1&=X^4    & F_5&=X^2YZ  & F_9&=XYZ^2   & F_{13}&=Y^2Z^2 \\
	F_2&=X^3Y   & F_6&=X^2Z^2 & F_{10}&=XZ^3 & F_{14}&= YZ^3 \\
	F_3&=X^3Z   & F_7&=XY^3   & F_{11}&=Y^4  & F_{15}&=Z^4 \\
	F_4&=X^2Y^2 & F_8&=XY^2Z  & F_{12}&=Y^3Z 
	\end{align*}
	The characters associated to the monomials above are the following:
	\begin{align*}
	\theta_1&=4l_1     & \theta_5&=2l_1+l_2+l_3  & \theta_9&=l_1+l_2+2l_3   & \theta_{13}&=2l_2+2l_3 \\
	\theta_2&=3l_1+l_2 & \theta_6&=2l_1+2l_3 & \theta_{10}&=l_1+3l_3 & \theta_{14}&= l_2+3l_3 \\
	\theta_3&=3l_1+l_3 & \theta_7&=l_1+3l_2   & \theta_{11}&=4l_2  & \theta_{15}&=4l_3 \\
	\theta_4&=2l_1+2l_2& \theta_8&=l_1+2l_2+l_3  & \theta_{12}&=3l_2+l_3
	\end{align*}
	Define $\phi:\{1,2,\dots,6\}\arr\{1,2,\dots,15\}$ as the unique function such that $Q_i^2=F_{\phi(i)}$. Then:
	\[ [\pi(Q_i)]=\prod_{k\neq \phi(i)} (h_4+\theta_k). \]
	We can now apply the localization formula to compute the generators of $\im(\pi_*)$, that is:
	\begin{equation*}
	\pi_*h_2^d = \sum_{i=1}^6 \frac{(h_2|_{[Q_i]})^d\cdot [\pi(Q_i)]}{c_{top}^T(T\PP(W_2)_{[Q_i]})}
			   = \sum_{i=1}^6 \frac{(-\chi_i)^d\cdot \left(\prod_{k\neq \phi(i)} (h_4+\theta_k)\right)}{\prod_{k\neq i} (\chi_k -\chi_i)}.
	\end{equation*}  
	A direct computation with Mathematica shows that all these generators are contained in $I_{\wt{Z}}$.
\end{proof}

Let ${\rm Hilb}^2\PP(W_2)$ be the Hilbert scheme of $0$-dimensional subscheme in $\PP(W_2)$ of length $2$. This scheme inherits a $\GL_{3}$-action from $\PP(W_2)$, and consequently it also inherits a $T$-action, where $T\subset\GL_{3}$ is the usual subtorus of diagonal matrices.
\begin{lemma}\label{lemma:Hilb2=PSym2}
	Let $\Scal$ be the tautological rank $2$ bundle over ${\rm Gr}_1(\PP(W_2))$, the Grassmannian of projective lines in $\PP(W_2)$. Then we have an equivariant morphism:
	\[ {\rm Hilb}^2\PP(W_2)\simeq \PP(\Sym^2\Scal^{\vee}). \]
	In particular, the scheme ${\rm Hilb}^2\PP(W_2)$ is smooth and its $\GL_{3}$-equivariant Chow ring is a free module over $A^*_{\GL_3}$.
\end{lemma}
\begin{proof}
	The idea of the proof is the following: given any $0$-dimensional subscheme $Z$ of length $2$, there exists a unique line $F$ containing $Z$. Moreover, there exists a unique quadratic form on $F$ that vanishes on $Z$: these two data, the line and the quadratic form defined on it, completely determine the subscheme $Z$, and vice versa.
	
	We now proceed to give a formal proof of the Lemma: let $\Zcal\subset {\rm Hilb}^2\PP(W_2)\times\PP(W_2)$ be the universal subscheme, and let $\Ical$ be the sheaf of ideals associated to $\Zcal$. Consider the two projections
	\[ \xymatrix{
	& {\rm Hilb}^2\PP(W_2)\times\PP(W_2) \ar[dl]^{\pr_1} \ar[dr]_{\pr_2} & \\
	{\rm Hilb}^2\PP(W_2) & & \PP(W_2)}
	\]
	An easy argument involving the cohomology and base change theorem shows that the sequence:
	\[ 0\arr \pr_{1*}(\Ical\otimes\pr_2^*\Ocal(1))\arr \pr_{1*}\pr_2^*\Ocal(1)\arr \pr_{1*}(\Ocal_{\Zcal}\otimes\pr_2^*\Ocal(1))\arr 0\]
	is an exact sequence of locally free sheaves. 
	
	The sheaf in the middle is isomorphic to $W_2^{\vee}\otimes \Ocal$ and the one on the right has rank $2$, hence we get a well defined morphism $\varphi:{\rm Hilb}^2\PP(W_2)\arr {\rm Gr}_1(\PP(W_2))$ such that $\varphi^*\Scal\simeq \pr_{1*}(\Ocal_{\Zcal}\otimes\pr_2^*\Ocal(1))^{\vee}$.
	
	Consider the morphism of locally free sheaves:
	\[ \Sym^2(\pr_{1*}(\Ocal_{\Zcal}\otimes\pr_2^*\Ocal(1))) \arr \pr_{1*}(\Ocal_{\Zcal}\otimes\pr_2^*\Ocal(2)) \]
	This is a surjective morphism of locally free sheaves from a rank $3$ locally free sheaf to a rank $2$ locally free sheaf. Therefore, the kernel has rank $1$ and its embedding in $\Sym^2(\pr_{1*}(\Ocal_{\Zcal}\otimes\pr_2^*\Ocal(1)))\simeq \varphi^*\Scal^{\vee}$ determines a morphism $\psi:{\rm Hilb}^2\PP(W_2)\arr \PP(\Sym^2\Scal^{\vee})$.
	
	The fact that $\psi$ is an isomorphism follows from the Zariski Main Theorem, as the morphism $\psi$ is birational and finite onto a smooth scheme.
\end{proof}
	
Consider the proper morphism $f:{\rm Hilb}^2\PP(W_2)\arr\PP(W_4)$ that is obtained by composing the Hilbert-Chow morphism ${\rm Hilb}^2\PP(W_2)\to\Sym^2\PP(W_2)$ with the multiplication map $\Sym^2\PP(W_2)\to\PP(W_4)$.

\begin{propos}\label{propos.Z22open}
	The ideal $\im(f_*)$ is contained in $I_{\wt{Z}}$.
\end{propos}

\begin{proof}
	From Lemma \ref{lemma:Hilb2=PSym2} we have an equivariant isomorphism ${\rm Hilb}^2\PP(W_2)\simeq\PP(\Sym^2\Scal^{\vee})$, where $\Scal$ is the tautological rank $2$ sheaf over the Grassmannian ${\rm Gr}_1(\PP(W_2))$.

	Let $\Qcal$ be the tautological \emph{quotient} sheaf over ${\rm Gr_1}(\PP(W_2))$, and let $\sigma_i$ be the $i^{\rm th}$ Chern class of $\Qcal$. Set $\tau:=c_1^{\GL_3}(\Ocal_{\PP(\Sym^2\Scal^{\vee})}(1))$. Consider the $A^*_{GL_3}$-module:
	\[ M:=A^*_{GL_3}\otimes\Sym^2\langle 1,\sigma_1,\sigma_2,\sigma_3,\sigma_4 \rangle\otimes \Sym^2\langle 1,\tau \rangle. \]
	The plain Chow ring of the Grassmannian is well known and the Chow groups of the Grassmannian are equivariantly formal. In particular, the natural morphism $M\arr A^*_{\GL_{3}}(\PP(\Sym^2\Scal^{\vee}))$ is surjective.
	
	We have in this way reduced the computation of $\im(f_*)$ to the computation of $f_*\xi$ for every generator $\xi$ of the $A^*_{GL_{3}}$-module $M$. We will do these computations by localization.
	
	Recall the notation introduced in the proof of Proposition \ref{propos.Z2} for the monomials representing $T$-invariant points of $\PP(W_2)$:
	\begin{eqnarray*}
		Q_1=X^2 & Q_4=Y^2 \\
		Q_2=XY & Q_5=YZ \\
		Q_3= XZ & Q_6=Z^2
	\end{eqnarray*}
	Let $T$ be the subtorus of diagonal matrices in $\GL_3$: then the fixed locus of ${\rm Hilb}^2\PP(W_2)$ consists of $45$ fixed points, which can be divided into two distinct classes
	\begin{itemize}
		\item The first class is formed by those subschemes supported on two distinct $T$-fixed points of $\PP(W_2)$. We will use the notation $(Q_i,Q_j)$ to indicate the subscheme supported on the points $[Q_i]$ and $[Q_j]$. There are $15$ fixed points of this type. \\
		\item The second class is formed by those subschemes supported on only one $T$-fixed point of $\PP(W_2)$ and schematically contained in a line joining another $T$-fixed point. We will use the notation $(Q_i>Q_j)$ to indicate the subscheme supported on the point $[Q_i]$ and schematically contained in the line that joins $[Q_i]$ and $[Q_j]$. There are $30$ fixed points of this type.\\
	\end{itemize} 
	To apply the localization techniques, we need the following ingredients:
	\begin{itemize}
		\item The equivariant top Chern class of the tangent space of ${\rm Hilb}^2\PP(W_2)$ at each fixed point.\\
		\item The fundamental class of the fixed points of $\PP(W_4)$ that are in the image of $f$.\\
		\item An explicit expression for the restriction of $\tau$, $\sigma_1$, $\sigma_2$, $\sigma_3$ and $\sigma_4$ to the equivariant Chow ring of each fixed point.\\
	\end{itemize}
	Let $\chi_i$ indicate the character associated to the monomial $Q_i$, as in the proof of Proposition \ref{propos.Z2}.
	
	Let $(Q_i,Q_j)$ be a fixed point of the first class, i.e. it represents the closed and reduced subscheme supported on the points $[Q_i]$ and $[Q_j]$ of $\PP(W_2)$. We have:
	\[ T({\rm Hilb}^2(\PP(W_2)))_{(Q_i,Q_j)}\simeq T(\PP(W_2))_{[Q_i]}\oplus T(\PP(W_2))_{[Q_j]}. \]

	This implies that:
	\[ c_{top}^T(T{\rm Hilb}^2\PP(W_2)_{(Q_i,Q_j)})=\prod_{k\neq i} (\chi_k - \chi_i) \prod_{l\neq j} (\chi_l - \chi_j). \]
	
	Let $(Q_i>Q_j)$ be a fixed point of the second type. If we restrict to the affine open subscheme $U$ of $\PP(W_2)$ where $a_i\neq 0$, the ideal associated to $(Q_i>Q_j)$ is the following:
	\[ I_{(Q_i>Q_j)}=\left( \frac{a_j^2}{a_i^2},\frac{a_k}{a_i} \right)_{k\neq j,i}. \]
	Let $R$ be ring of functions of $U$. Then from the usual description of the tangent space of the Hilbert scheme, we deduce that:
	\begin{align*} T{\rm Hilb}^2(\PP(W_2))_{(Q_i>Q_j)}&=\Hom_{R/I}(I/I^2,R/I)\\
	&=\Big\langle \frac{a_j^2}{a_i^2}^\vee\otimes 1,\frac{a_k}{a_i}^{\vee}\otimes 1,\frac{a_j^2}{a_i^2}^\vee\otimes \frac{a_j}{a_i},\frac{a_k}{a_i}^\vee \otimes \frac{a_j}{a_i} \Big\rangle_{k\neq i,j}.
	\end{align*}
	From this we deduce:
	\[ c_{top}^T(T{\rm Hilb}^2\PP(W_2)_{(Q_i>Q_j)})=2(\chi_i-\chi_j)^2\prod_{k\neq i,j} (\chi_k - \chi_i)(\chi_k-\chi_j). \]
	The next step is the computation of the classes $[f((Q_i,Q_j))]$ and $[f((Q_i>Q_j))]$ in the equivariant Chow ring of $\PP(W_4)$. We have $f((Q_i,Q_j))=[Q_iQ_j]$ and $f((Q_i>Q_j))=[Q_i^2]$.
	
	Just as in the proof of Proposition \ref{propos.Z2}, denote:
	\begin{align*}
	F_1&=X^4    & F_5&=X^2YZ  & F_9&=XYZ^2   & F_{13}&=Y^2Z^2 \\
	F_2&=X^3Y   & F_6&=X^2Z^2 & F_{10}&=XZ^3 & F_{14}&= YZ^3 \\
	F_3&=X^3Z   & F_7&=XY^3   & F_{11}&=Y^4  & F_{15}&=Z^4 \\
	F_4&=X^2Y^2 & F_8&=XY^2Z  & F_{12}&=Y^3Z 
	\end{align*}
	and let $\theta_i$ be the character associated to the monomial $F_i$.
	Let $\mu:\{1,2,\dots,6\}^{\times 2}\arr \{1,2,\dots,15\}$ be the function such that $f((Q_i,Q_j))=F_{\mu(i,j)}$ and $f((Q_i>Q_j))=F_{\mu(i,i)}$ for $i\neq j$.
	Then we have:
	\begin{align*}
	[f((Q_i,Q_j))]&=\prod_{k\neq \mu(i,j)} (h_4+\theta_k), &[f((Q_i>Q_j))]&=\prod_{l\neq \mu(i,i)} (h_4+\theta_l).
	\end{align*}
	
	Next we study the restriction of $\Qcal$, the pullback to ${\rm Hilb}^2(\PP(W_2))$ of the quotient tautological sheaf over ${\rm Gr}_1(\PP(W_2))$, to the fixed points.	
	Take a fixed point of the first class, say $(Q_i,Q_j)$ with $i\neq j$: then $\Qcal|_{(Q_i,Q_j)}$ is equal by construction to the $4$-dimensional quotient vector space $W_2/\langle Q_i,Q_j \rangle$. We deduce that $c_p^T(\Qcal|_{(Q_i,Q_j)})=\sym_p(\{\chi_k\}_{k\neq i,j})$, where $\sym_p(\cdot)$ denotes the elementary symmetric polynomial of degree $p$.
	
	Exactly the same formula holds for the fixed points $(Q_i>Q_j)$ of the second class.
	
	The last ingredient missing is the computation of the first Chern class of the line bundle $\Ocal_{\PP(\Sym^2\Scal^{\vee})}(1)$ restricted to the fixed points. Let $F_{ij}$ be the unique line in $\PP(W_2)$ joining $[Q_i]$ and $[Q_j]$. We have:
	\[ \Ocal_{\PP(\Sym^2\Scal^{\vee})}(-1)|_{(Q_i,Q_j)}\simeq  \ker \left(\Sym^2 H^0(F_{ij},\Ocal_{F_{ij}}(1))\arr H^0(Z_{ij},\Ocal_{F_{ij}}(2)|_{Z_{ij}})\right) \]
	where $Z_{ij}$ is the subscheme represented by the point $(Q_i,Q_j)$.
	
	The vector space $H^0(F_{ij},\Ocal_{F_{ij}}(1))$ is generated by the restriction of the global sections $a_i$ and $a_j$ of $H^0(\PP(W_2),\Ocal(1))$. We deduce that $\Ocal_{\PP(\Sym^2\Scal^{\vee})}(-1)|_{(Q_i,Q_j)}$ is generated by the unique (up to scalar) quadratic form in $a_i$ and $a_j$ that vanishes on $Z_{ij}$, i.e. the form $a_i\cdot a_j$.
	
	Henceforth, the rank $1$ vector space $\Ocal_{\PP(\Sym^2\Scal^{\vee})}(1)|_{(Q_i,Q_j)}$ is generated by $(a_i\cdot a_j)^{\vee}$ and we have $c_1^T(\langle (a_i\cdot a_j)^{\vee}\rangle )=\chi_i+\chi_j$.
	
	A similar argument shows that $c_1^T(\Ocal_{\PP(\Sym^2\Scal^{\vee})}(1)|_{(Q_i>Q_j)})=2\chi_i$.
	
	We are ready to use the localization techniques to prove that $\im(f_*)\subset I_{\wt{Z}}$. The former ideal is generated by the elements
	\begin{multline*}
		f_*(\tau^a \sigma_1^b \sigma_2^c \sigma_3^c \sigma_4^d)=\sum_{(Q_i,Q_j)} \frac{(\tau^a \sigma_1^b \sigma_2^c \sigma_3^c \sigma_4^d)|_{(Q_i,Q_j)}}{c_{top}^T(T{\rm Hilb}^2(\PP(W_2))_{(Q_i,Q_j)})}[f((Q_i,Q_j))] \\ + \sum_{(Q_i>Q_j)} \frac{(\tau^a \sigma_1^b \sigma_2^c \sigma_3^c \sigma_4^d)|_{(Q_i>Q_j)}}{c_{top}^T(T{\rm Hilb}^2(\PP(W_2))_{(Q_i>Q_j)})}[f((Q_i>Q_j))]
	\end{multline*}
	where the expression $\tau^a \sigma_1^b \sigma_2^c \sigma_3^c \sigma_4^d$ is in the image of $A_T^*\otimes M\arr A_T^*({\rm Hilb}^2\PP(W_2))$.
	
	A direct computation with Mathematica shows that each of these cycles is in $I_{\wt{Z}}$.
\end{proof}

Observe that the proper morphism $f:{\rm Hilb}^2\PP(W_2)\arr\PP(W_4)$ restricts to an isomorphism over $Z_{\{2,2\}} \cup Z_{1,1,2} \cup Z_{1,2}^{(2,1)}$. Applying Lemma \ref{basic.principle} we deduce that the ideal of cycles coming from the stratum $\ov{Z}_{\{2\}}$ is equal to the sum of the image of the pushforward morphism
\[ f_*: A_*^{\GL_{3}}({\rm Hilb}^2\PP(W_2))\arr A_*^{\GL_{3}}(\PP(W_4)) \]
plus the ideal of cycles that come from $\ov{Z}_{\{1,1,1,1\}}$ and $\ov{Z}_{\{2\}}^{(2)}$.

We already know from Proposition \ref{propos.Z1111} that the cycles coming from $\ov{Z}_{\{1,1,1,1\}}$ are contained in $I_{\wt{Z}}$. By Proposition \ref{propos.Z22open} and Proposition \ref{propos.Z2} the same thing holds for the image of $f_*$ and for the stratum $\ov{Z}_{\{2\}}^{(2)}$. Therefore, we have proved Proposition \ref{propos.Z22}.

\subsection{The push-forward of classes from $\ov{Z}_{\{1,3\}}$ is in the ideal $(I_{\wt{Z}}, \delta_{\{ 1,3 \}})$.}

First, we determine the class $\delta_{\{1,3\}}$. By using the formula from \cite[Proposition 3.4]{Fu-Vi}, after straightforward computations, we get
\[
\delta_{\{ 1,3 \}} = 55 h_4^3 - 220 c_1 h_4^2 + (280 c_1^2 + 40 c_2) h_4 + (224 c_3- 96 c_1^3 - 128 c_1 c_2).
\]

\begin{lemma}
The classes $\alpha_1, \alpha_2, \alpha_3, \delta_{\{ 1, 3 \}}$ are a set of independent generators for the ideal $(\alpha_1, \alpha_2, \alpha_3, \delta_{\{ 1, 3 \}})$.
\end{lemma}

\begin{proof}
By performing computations mod 3 and by arguing as in Lemma \ref{alpha.classes.independent}, we are reduced to check that the class $\delta_{\{ 1, 3 \}}$ mod 3 is not an integral multiple of the class $\alpha_3$ mod 3. A simple computation shows that this is not possible.
\end{proof}

\begin{rmk}
Since the restriction map $\pi_1: \wt{Z}|_{\ov{Z}_{\{ 1, 3\}}} \to \ov{Z}_{\{ 1, 3\}}$ is generically 3:1, we expect the inclusion $3\delta_{\{ 1,3\}} \in I_{\wt{Z}}$. We can verify this directly by checking that
\[
3\delta_{\{ 1,3\}} = \left( 9 h_4^2 - 20 c_1 h_4 + 8 c_1^2 \right)\alpha_1 - 2(3h_4 - 8c_1)\alpha_2 -24\alpha_3.
\]
\end{rmk}

We are ready to prove the main result of this Subsection, which is the following:
\begin{propos}\label{propos.Z13}
	The image of $i_*:A_{\GL_3}^*(\ov{Z}_{\{1,3\}})\arr A_{\GL_3}^*(\PP(W_4))$ is contained in $(I_{\wt{Z}},\delta_{\{ 1, 3 \}})$.
\end{propos}
\begin{proof}
Consider the multiplication morphism
\[ \pi_{\{1,3\}}:\PP(W_1)\times\PP(W_3)\arr \PP(W_4).  \]
Observe that $\pi_{\{1,3\}}^*(h_4)=h_1+h_3$, where $h_i$ denotes the hyperplane section of $\PP(W_i)$. This simple fact, together with the projection formula, implies that the image of $\pi_{\{1,3\}*}$ is generated by $\delta_{\{ 1, 3 \}}$, $\pi_{\{1,3\}*}(h_1)$ and $\pi_{\{1,3\}*}(h_1^2)$.

These last two generators can be computed by localization. Set $$ \chi_{ijk}:= il_1 + jl_2+ kl_3 .$$
Then we have:
\begin{itemize}
	\item $c_{top}^T(T(\PP(W_1)\times\PP(W_3))_{[X^iY^jZ^k],[X^lY^mZ^n]})=\prod (\chi_{abc}-\chi_{ijk})\prod (\chi_{def}-\chi_{lmn})$, where the product is taken over the triples of natural numbers $(a,b,c)\neq (i,j,k)$ and $(d,e,f)\neq (l,m,n)$ such that $a+b+c=1$, $d+e+f=3$. \\
	\item $[\pi_{\{1,3\}}([X^iY^jZ^k],[X^lY^mZ^n])]=\prod (h_4+\chi_{abc})$, where the product is taken over the triples of natural numbers $(a,b,c)\neq (i+l,j+m,k+n)$ such that $a+b+c=4$.\\
	\item $h_1^d|_{A^*_T([X^iY^jZ^k],[X^lY^mZ^n])}=(-\chi_{ijk})^d$.\\
\end{itemize}
The localization formula tells us the following:
\[ \pi_{\{1,3\}*}h_1^d=\sum \frac{(h_1^d|_{A^*_T([X^iY^jZ^k],[X^lY^mZ^n])})\cdot[\pi_{\{1,3\}}([X^iY^jZ^k],[X^lY^mZ^n])]}{c_{top}^T(T(\PP(W_1)\times\PP(W_3))_{[X^iY^jZ^k],[X^lY^mZ^n]})}.  \]
With the aid of Mathematica, one can explicitly write down the expressions above in terms of $h_4$, $c_1$, $c_2$ and $c_3$, and then check that both cycles are contained in $(\alpha_1,\alpha_2,\alpha_3,\delta_{\{1,3\}})$.
\end{proof}

Putting together Proposition \ref{propos.Z22} and Proposition \ref{propos.Z13} we deduce the following intermediate result.

\begin{theorem}\label{main.result.reducible.quartics}
Assume that the base field has characteristic different from 2 and 3. Then
\[
i_*\left( A^{\GL_3} _{*} ( \ov{Z}_\reducible) \right) \subset \left( \alpha_1,\alpha_2,\alpha_3, \delta_{\{ 1,3 \}} \right)
\]

where

\begin{eqnarray*}
\alpha_1 &=& 27 h_4 - 36 c_1,\\ 
\alpha_2 &=& 9 h_4^2 - 6 c_1 h_4 - 24 c_2,\\
\alpha_3 &=& h^3 - c_1 h^2 + c_2 h - 28 c_3,\\
\delta_{\{ 1,3 \}} &=& 55 h_4^3 - 220 c_1 h_4^2 + (280 c_1^2 + 40 c_2) h_4 + (224 c_3- 96 c_1^3 - 128 c_1 c_2).
\end{eqnarray*}
\end{theorem}

\section{Irreducible quartics}\label{sec.irr}

We now stratify the space of irreducible quartics by number of singular points.

\subsection{Quartics with two singular points.} \label{subsec.twosingpoints}

Let $H:=\hilb^{2}_{\ptwo}$ be the Hilbert scheme of 0-dimensional subschemes of $\ptwo$ of length 2, and let $p$ be a point of $H$. If we call $\Sigma(p)$ the corresponding 0-scheme of length 2 in $\ptwo$, the ideal $I \left( \Sigma(p) \right)$ is generated by a linear form $L$ and a quadratic form $Q$ (which is not a multiple of $L$). 

Let $l$ be the line of equation $L=0$ and let $q$ be the conic of equation $Q=0$. Notice that the line $l$ is uniquely determined by $p$ and this implies that there is a well defined morphism $\phi: H \to \PP(W_1)$. On the other hand, the conic $q$ is in the equivalence class of any quadratic form whose equation can be written as a linear combination of $Q$ and $LF$, where $F$ is any linear form.

We see the line $l$ as a point in $\Pro(W_1)$. Let $L$ be a linear form in three variables such that $[L]=l$. The linear form $L$ can be seen as a (non-zero) element of the fiber of the tautological bundle $\Ocal_{\PP(W_1)}(-1)$ over $l$.

Summarizing, we get the following exact sequence of locally free sheaves over $\PP(W_1)$:
\begin{equation}\label{es.hilb2}
0 \to \cO_{\PP(W_1)}(-1) \otimes W_1 \xrightarrow{\varphi} \cO_{\PP(W_1)} \otimes W_2 \to V \to 0
\end{equation}
where $\varphi$ is the morphism defined on the fiber over $l$ as
\[
\varphi(l, L \otimes F) = (l, LF ).
\]
The quotient sheaf $V$ is a locally free sheaf of rank 3. Notice that $V$ parametrizes the pairs $(l, [Q]_l)$ where $l$ is a line in $\ptwo$ and $[Q]_l$ is the equivalence class of all quadratic forms up to multiples of any linear form associated to $l$. Exactly the same argument used to prove Lemma \ref{lemma:Hilb2=PSym2}, after substituting $\PP(W_2)$ with $\PP^2$, shows the following:

\begin{lemma}\label{lemma.Hilb2bis}
	We have an isomorphism of $\GL_3$-equivariant projective bundles:
	\[
	\xymatrix{
		\PP(V) \ar[rr]^{\cong} \ar[rd] & & H \ar[ld]_{\phi}\\
		& \PP(W_1) &
	}
	\]
	In particular $A^*_{\GL_3}(H)$ is freely generated as $A^*_{\GL_3}$-module by the elements $h_1^is^j$, where $h_1:=c_1(\Ocal_{\PP(W_1)}(1))$, $s:=c_1(\Ocal_{\Pro(V)}(1))$ and $0\leq i \leq 2$, $0\leq j \leq 2$.
\end{lemma} 

\begin{propos}\label{def.D}
There is a projective subbundle $D \subseteq \PP(W_{4}) \times H$ over $H$ whose geometric points correspond to pairs $(p, C)$, where $\Sigma(p) \subseteq \ptwo$ is a 0-dimensional subscheme of length 2, and $C \subset \ptwo$ is a quartic curve whose sheaf of ideals contains $I \left( \Sigma(p) \right)^{2}$. \end{propos}

\begin{proof}
Let $E$ be the standard representation of $\GL_3$ and define $\PP^2:=\PP(E)$. Let $I \subseteq \PP(W_1) \times \PP^2$ be the incidence relation. The projection $I \arr \PP(W_1)$ makes $I$ into a $\PP^{1}$-bundle on $\PP(W_1)$. More precisely, we have an isomorphism:
\[
I \cong \PP\bigl(\cotd\bigr).
\]
The embedding $I \subseteq \PP(W_1) \times \ptwo$ is induced by the natural $\GL_{3}$-equivariant embedding
\[
\cotd \subseteq E \otimes \Ocal_{\PP(W_1)}
\]
of locally free sheaves on $\PP(W_1)$.

The Hilbert scheme $H$ can be alternatively described as the relative Hilbert scheme $\hilb^{2}_{I/\PP(W_1)}$ of length two subschemes of the fibers of the projection 
\[ I = \PP \left( \Omega_{\PP(W_1)}(1) \right) \arr \PP(W_1).\]
In other words, the natural map $\hilb^{2}_{I/\PP(W_1)} \to H$, induced by the projection $I \to \PP^2$, is an isomorphism, essentially because every length two subscheme of $\ptwo$ is contained in a unique line.

The universal family $S \subseteq H\times_{\PP(W_1)}I \subseteq H \times \ptwo$ can be described as follows: consider the projectivization $\PP\bigl(\stwo\bigr)$ of the sheaf of quadratic forms on $\cotd$. Then 
\begin{equation}
S \subseteq I \times_{\PP(W_1)}H \cong \PP(\cotd) \times_{\PP(W_1)} \PP\bigl(\stwo\bigr)
\end{equation}
is the universal hypersurface of degree~$2$.

Denote by $S_{1}$ the first infinitesimal neighborhood of $S$ in $H \times \ptwo$. If $\Sigma \subseteq \PP^{2}$ is a length two subscheme of $\ptwo$ corresponding to a point $p$ of $H$, the fiber of $S_{1} \arr H$ over $p$ is the subscheme of $\ptwo$ with sheaf of ideals $I(\Sigma)^{2}$. It is easy to see that the dimension over $k(p)$ of $\H^{0}\bigl(S_{1}(p), \cO(4)\bigr)$, where $S_{1}(p)$ is the fiber of $S_{1}$ over $p$, is always $6$; hence the projection $S \arr H$ is finite and flat of degree~$6$.

Consider the two projections
\[
\xymatrix{
H \times \ptwo \ar[r]^{\pr_2} \ar[d]_{\pi} & \ptwo \\
H & \\
}
\]

Let $F$ be the pushforward $\pi_{*}\bigl(I(S)^{2} \otimes \pr_{2}^{*}\cO_{\ptwo}(4)\bigr)$. The pushforward to $H$ of the exact sequence
   \[
   0 \arr I(S)^{2} \otimes \pr_{2}^{*}\cO_{\ptwo}(4) \arr
   \pr_{2}^{*}\cO_{\ptwo}(4) \arr \cO_{S_{1}}\otimes\pr_{2}^{*}\cO_{\ptwo}(4) \arr 0
   \]
of sheaves on $H \times \ptwo$ gives an exact sequence
   \[
   0 \arr F \arr W_{4}\otimes\cO_{H} \arr \pi_{*}\cO_{S_{1}}(4)
   \]
where by $\cO_{S_{1}}(4)$ we denote the restriction of $\pr_{2}^{*}\cO_{\ptwo}(4)$ to $S_{1}$. 

A straightforward argument shows that if $\Sigma \subseteq \ptwo$ is a length two subscheme, the restriction homomorphism $\H^{0}\bigl(\ptwo, \cO_{\ptwo}(4)\bigr) \arr \H^{0}\bigl(\ptwo, (\cO_{\Sigma}/I(\Sigma)^2)(4)\bigr)$ is surjective. Hence the homomorphism $W_{4}\otimes\cO_{H} \arr \pi_{*}\cO_{S_{1}}(4)$ is surjective. Since $\pi_{*}\cO_{S_{1}}(4)$ is a rank~$6$ locally free sheaf, we have that $F$ is a locally free sheaf of rank $9$.

Set $D := \PP(F) \subseteq \PP(W_{4}) \times H$; then $D$ is a smooth integral variety of dimension $12$. Its geometric points correspond to pairs $(p, C)$, where $\Sigma(p) \subseteq \ptwo$ is a length two subscheme, and $C \subset \ptwo$ is a quartic whose sheaf of ideals contains $I(\Sigma(p))^{2}$. So, generically, $D$ parametrizes quartics with exactly two singular points.
\end{proof}

The variety $D$ possesses two desirable features: first, the projection $D\arr \PP(W_4)$ factorizes by construction through $\ov{Z}_{\rm binod}$, the closed subscheme of quartics having at least two singular points, and it is 1:1 over $Z_{\rm binod}$. Second, we can easily determine the generators of $A^*_{\GL_{3}}(D)$ as $A^*_{\GL_{3}}$-module, as it is a projective bundle over $H$.

We have seen in Lemma \ref{lemma.Hilb2bis} that the equivariant Chow ring of $H$ is generated, as an $A^*_{\GL_{3}}$-module, by the elements $h_1^is^j$, where $h_1:=c_1(\Ocal_{\PP(W_1)}(1))$, $s:=c_1(\Ocal_{\Pro(V)}(1))$ and $0\leq i \leq 2$, $0\leq j \leq 2$.
Consider the diagram
$$\xymatrix{
	& D \ar[dl]_{p} \ar[dr]^{q} & \\
	\Pro(W_4)& &H}$$
Then the following Lemma easily follows by applying the projection formula.
\begin{lemma}
	The image of $p_*:A^*_{\GL_{3}}(D)\to A^*_{\GL_{3}}(\Pro(W_4))$ is generated as an ideal by the elements 
	\[
	\gamma_{ij}:=p_*q^*h_1^is^j,
	\]
	where $0\leq i \leq 2$, $0\leq j \leq 2$.
\end{lemma}
We can use the localization formula to explicitly compute the elements $\gamma_{ij}$, so to prove the following.
\begin{propos}\label{propos.twonodes}
	The image of $p_*:A^*_{\GL_{3}}(D)\to A^*_{\GL_{3}}(\Pro(W_4))$ is contained in $I_{\wt{Z}}$.
\end{propos}
\begin{proof}
	Let $T\subset \GL_{3}$ be the maximal torus of diagonal matrices. Then the $T$-fixed locus of $H$ consists of the nine points $p_1,\dots,p_9$ parametrizing the subschemes of $\ptwo$ whose associated homogeneous ideals are:
	$$\begin{matrix*}
	I_{p_1}=(X,YZ)&I_{p_2}=(Y,XZ)&I_{p_3}=(Z,XY) \\
	I_{p_4}=(X,Y^2)&I_{p_5}=(X^2,Y)&I_{p_6}=(Y,Z^2)\\
	I_{p_7}=(Z,Y^2)&I_{p_8}=(X,Z^2)&I_{p_9}=(X^2,Z)
	\end{matrix*}$$
	From this we see that the $T$-fixed locus of $D$ consists of the pairs $F=(C,p)$ where $p$ is a $T$-fixed point of $H$ and $C$ is a quartic whose defining equation is a monomial contained in $I_p^2$.	
	We need to compute $c^T_{top}(TD_F)$ for every fixed point $F=(C,p)$ of $D$. Observe that $$c_{top}^T(TD_F)=c_{top}^T(TH_p)c_{top}^T(T\Pro(V_p)_C)$$ where $\Pro(V_p)\subset\Pro(W_4)$ is the subspace of quartics containing the first infinitesimal neighborhood of $\Sigma(p)$. We are going to compute separately the two top Chern classes above.
	
	Recall that $TH_p=\Hom(\Ical_p/\Ical_p^2,\Ocal_p)$. Suppose that the homogeneous ideal associated to $p$ is $(X,YZ)$: then the subscheme $\Sigma(p)$ is the union of the two reduced points $q=[0:1:0]$ and $r=[0:0:1]$ and we have:
	$$ \Hom(\Ical_p/\Ical_p^2,\Ocal_p)=\Hom \left( \left \langle \dfrac{X}{Y},\dfrac{Z}{Y} \right \rangle,k \right )\oplus\Hom \left( \left \langle \dfrac{X}{Z},\dfrac{Y}{Z} \right \rangle,k \right). $$
	Recall that, once a base $\{e_1,\dots,e_n\}$ of a vector space $V$ is fixed, there is a well defined base $\{e_1^\vee,\dots,e_n^\vee\}$ of dual elements for $V^\vee$.
	Then we can say that:
	$$ TH_p=\left \langle \dfrac{X}{Y}^\vee,\dfrac{Z}{Y}^\vee,\dfrac{X}{Z}^\vee,\dfrac{Y}{Z}^\vee \right \rangle. $$
	From this we immediately deduce:
	$$ c_{top}^T(TH_{p_1})=(l_2-l_1)(l_2-l_3)(l_3-l_1)(l_3-l_2). $$
    With the same arguments we deduce:
	\begin{align*}
	c_{top}^T(TH_{p_2})&=(l_3-l_1)(l_3-l_2)(l_1-l_2)(l_1-l_3); \\
	c_{top}^T(TH_{p_3})&=(l_2-l_1)(l_2-l_3)(l_1-l_2)(l_1-l_3).
	\end{align*}
	Consider now the subscheme associated to $p_4$, for which we have:
	$$ \Hom ( \Ical_{p_4}/\Ical_{p_4}^2,\Ocal_{p_4}  )=\Hom \left( \left \langle \dfrac{X}{Z},\dfrac{Y^2}{Z^2} \right \rangle, \left \langle 1,\dfrac{Y}{Z} \right \rangle \right). $$
	From this we deduce
	$$ TH_{p_4}=\left \langle \dfrac{X}{Z}^\vee\otimes 1,\dfrac{Y^2}{Z^2}^\vee \otimes 1, \dfrac{X}{Z}^\vee\otimes \dfrac{Y}{Z},\dfrac{Y^2}{Z^2}^\vee \otimes \dfrac{Y}{Z} \right \rangle $$
	so that we obtain 
	$$c_{top}^T(TH_{p_4})=(l_3-l_1)(2l_3-2l_2)(l_2-l_1)(l_3-l_2).$$
	The remaining Chern classes are computed in a similar way. The final result is:
	\begin{align*}
	c_{top}^T(TH_{p_5})&=2(l_3-l_2)(l_1-l_2)(l_3-l_1)(l_3-l_1);\\
	c_{top}^T(TH_{p_6})&=2(l_1-l_2)(l_3-l_2)(l_1-l_3)(l_1-l_3);\\
	c_{top}^T(TH_{p_7})&=2(l_1-l_3)(l_2-l_3)(l_1-l_2)(l_1-l_2);\\
	c_{top}^T(TH_{p_8})&=2(l_2-l_1)(l_3-l_1)(l_2-l_3)(l_2-l_3);\\
	c_{top}^T(TH_{p_9})&=2(l_2-l_3)(l_1-l_3)(l_2-l_1)(l_2-l_1).
	\end{align*}
	Let ${(X^iY^jZ^k)^\vee}$ be homogeneous coordinates for $\Pro(W_4)$. Fix a point $p\in H$ with associated homogeneous ideal $I_p$ and a monomial $f=X^aY^bZ^c$ which is contained in $I_p^2$. Let $C:=\{f=0\}$. Then we have:
	$$ T\Pro(V_p)_C=\left \langle \dfrac{(X^iY^jZ^k)}{(X^aY^bZ^c)} \text{ where } X^iY^jZ^k\in I_p^2, (i,j,k)\neq(a,b,c) \right \rangle. $$
	Therefore:
	$$ c_{top}^T(T\Pro(V_p)_C)=\prod ((i-a)l_1+(j-b)l_2+(k-c)l_3) $$
	where the product is taken over all $(i,j,k)$ such that $i,j,k\geq0$, $i+j+k=4$, $(i,j,k)\neq (a,b,c)$ and $X^iY^jZ^k\in I_p^2$.
	
	Now we compute the class of the fixed point $[C]$ of $\Pro(W_4)$, where $C$ is defined as before. Observe that such a point is a complete intersection, defined by the homogeneous equations:
	$$ (X^iY^jZ^k)^\vee=0 \text{ for } (i,j,k)\neq (a,b,c).$$
	Applying \cite[Lemma 2.4]{EFh} we deduce:
	$$[C]=\prod (h_4+il_1+jl_2+kl_3)$$
	where the product is taken over all the $(i,j,k)$ such that $i,j,k\geq0$, $i+j+k=4$, $(i,j,k)\neq (a,b,c)$.
	
	We need then to compute the restriction of $h_1$ and $s$ to the fixed points of $H$. Recall that $h_1=c_1(\Ocal_{\Pro(W_1)}(1))$ and $s=c_1(\Ocal_{\Pro(V)}(1))$. Observe that the restriction of the line bundle $\Ocal_{\Pro(W_1)}(-1)$ to a point $p$ of $H$ is generated by an equation for the unique line passing through $\Sigma(p)$ or, in other terms, by the generator of degree $1$ of $I_p$. This implies:
	$$\begin{matrix*}
	c_1^T(\Ocal_{\Pro(W_1)}(1)(p_1))=-l_1 & c_1^T(\Ocal_{\Pro(W_1)}(1)(p_4))=-l_1 &  c_1^T(\Ocal_{\Pro(W_1)}(1)(p_7))=-l_3 \\
	c_1^T(\Ocal_{\Pro(W_1)}(1)(p_2))=-l_2  & c_1^T(\Ocal_{\Pro(W_1)}(1)(p_5))=-l_2 &  c_1^T(\Ocal_{\Pro(W_1)}(1)(p_8))=-l_1 \\
	c_1^T(\Ocal_{\Pro(W_1)}(1)(p_3))=-l_3 & c_1^T(\Ocal_{\Pro(W_1)}(1)(p_6))=-l_2  & c_1^T(\Ocal_{\Pro(W_1)}(1)(p_9))=-l_3
	\end{matrix*}$$

	Similarly, it is easy to see that the restriction of the line bundle $\Ocal_{\Pro(V)}(1)$ to a fixed point $p$ of $H$ is generated by an equation for the unique monomial of degree $2$ that generates $I_p$. This implies:
	$$\begin{matrix*}
	c_1^T(\Ocal_{\Pro(V)}(1)(p_1))=-l_2-l_3 & c_1^T(\Ocal_{\Pro(V)}(1)(p_4))=-2l_2 & c_1^T(\Ocal_{\Pro(V)}(1)(p_7))=-2l_2 \\
	c_1^T(\Ocal_{\Pro(V)}(1)(p_2))=-l_1-l_3  & c_1^T(\Ocal_{\Pro(V)}(1)(p_5))=-2l_1 &  c_1^T(\Ocal_{\Pro(V)}(1)(p_8))=-2l_3  \\
	c_1^T(\Ocal_{\Pro(V)}(1)(p_3))=-l_1-l_2  & c_1^T(\Ocal_{\Pro(V)}(1)(p_6))=-2l_3 & c_1^T(\Ocal_{\Pro(V)}(1)(p_9))=-2l_1
	\end{matrix*}$$
	Applying the projection formula we get:
	$$p_*q^*h_1^is^j=\sum \frac{c_1^T(\Ocal_{\Pro(W_1)}(1)(p))^i c_1^T(\Ocal_{\Pro(V)}(1)(p))^j}{c_{top}^T(TD_{C,p})}[C]$$
	where the sum is taken over all the fixed points $(C,p)$ of $D$. A straightforward computation with Mathematica concludes the proof.
\end{proof}
Again with the aid of Mathematica, from Proposition \ref{propos.twonodes} we deduce the following:
\begin{corol}\label{cor.twonodes}
	The image of $A^{\GL_{3}}_*(Z_{\rm binod})\arr A^{\GL_{3}}_*(\Pro(W_4)\smallsetminus \ov{Z}_{\rm trinod})$ is contained in the restriction to the latter of the ideal $(I_{\wt{Z}},\delta_{\{ 1, 3 \}})$.
\end{corol}

\subsection{Quartics with three singular points.} \label{subsec.threesingularpoints}

Let $H:=\hilb^3\ptwo$ be the Hilbert scheme of $0$-dimensional subschemes of $\ptwo$ of length $3$. It is well known (see for instance \cite{fog68}) that $H$ is a smooth scheme of dimension $6$ and that there exists a universal family
$$ \xymatrix{
	\Sigma \ar[d] \ar@{}[r]|-*[@]{\subset} & H\times\ptwo \\
	H }$$
Let $\Sigma_1$ be the first infinitesimal neighborhood of $\Sigma$, i.e. the closed subscheme of $H\times\ptwo$ defined by the ideal $I_\Sigma^2$.

Recall that a $0$-dimensional, closed subscheme $S$ of $\PP^2$ is said to be \emph{curvilinear} if every connected component of $S$ can be embedded in a plane curve.
We set 
\[ Y := 
\left\{
\begin{matrix}
[S]\in H \text{ such that } S\subset \PP^2 \text{ has dimension 0, length 3,} \\
	\text{ is supported on one point and it is not curvilinear}  
\end{matrix} \right\} \] 
This $Y$ is a closed subset of $H$; we take it with the reduced scheme structure. If $S$ is such that $[S]\in Y$, then locally $I_S=(x,y)^2$. Observe that $Y$ is isomorphic to $\ptwo$ and is regularly embedded in $H$. Notice that the morphism
$$ \xymatrix{
	\Sigma_1|_{H\setminus Y} \ar[d] \ar@{}[r]|-*[@]{\subset} & H\times\ptwo \\
		H\setminus Y }$$
is a flat morphism and its fibers are $0$-dimensional closed subschemes of $\ptwo$ of length $9$. 

On the other hand, the fiber of $\Sigma_1$ over a point in $Y$ corresponds to a $0$-dimensional closed subscheme of $\ptwo$ of length $10$, supported on one point, whose defining ideal is locally equal to $(x,y)^4$. Our goal is to modify $H$ so that the family $\Sigma_1|_{H\setminus Y}$ can be extended to a family which is flat over the whole base.

\begin{propos}\label{propos.Htilde}
	Let $\wt{H}$ be the blow-up of $H$ along $Y$. Then there exists a flat family $\wt{\Sigma}_1\arr\wt{H}$ of $0$-dimensional, length $9$ closed subschemes of $\PP^2$ that extends $\Sigma_1|_{H\setminus Y}\arr H\setminus Y$. 
\end{propos}

\begin{rmk}\label{rmk.fiber of extension}
	Let $(x,y)^2$ be a local description of the ideal of $S\subset \PP^2$, so that $[S]\in Y$: then the fiber of the normal bundle $N_{Y,H}$ at $[S]$ can be identified with the space of homogeneous forms $f$ of degree $3$ in $x$ and $y$, and the fiber of $\Sigma_1$ over a point $([S],[f])$ in the exceptional divisor of $\wt{H}$ corresponds to the subscheme of $\PP^2$ whose ideal can be locally described as $(f,(x,y)^4)$.
\end{rmk}

We postpone the proof of Proposition \ref{propos.Htilde} to Section \ref{sec.modifications}. Consider the following diagram:
$$\xymatrix{
	\wt{H}\times\ptwo \ar[r]^{\hspace{0.3cm}\pr_2} \ar[d]^{\pr_1} & \ptwo \\
	\wt{H} }$$
Then over $\wt{H}\times\ptwo$ we have the short exact sequence of sheaves:
$$ 0\arr I_{\wt{\Sigma}_1} \arr \Ocal_{\wt{H}\times\ptwo} \arr \Ocal_{\wt{\Sigma}_1} \arr 0 $$
We can twist it by $\pr_2^*\Ocal(4)$ and push it forward along $\pr_1$, obtaining in this way the exact sequence:
$$0 \arr \pr_{1*}(I_{\wt{\Sigma}_1}\otimes\pr_2^*\Ocal(4)) \arr \pr_{1*}\pr_2^*\Ocal(4) \arr \pr_{1*}(\Ocal_{\wt{\Sigma}_1}\otimes\pr_2^*\Ocal(4)) $$
We have $\pr_{1*}\pr_2^*\Ocal(4)\simeq\Ocal_{\wt{H}\times\ptwo}\otimes W_4$, therefore it is a free sheaf of rank $15$, and by Grauert theorem we deduce that $\pr_{1*}(\Ocal_{\wt{\Sigma}_1}\otimes\pr_2^*\Ocal(4))$ is locally free of rank $9$.

Observe that the fiber over a fixed point in $\wt{H}$ of the kernel of the morphism $$\varphi:\pr_{1*}\pr_2^*\Ocal(4) \arr \pr_{1*}(\Ocal_{\wt{\Sigma}_1}\otimes\pr_2^*\Ocal(4))$$ can be interpreted as the space of ternary forms of degree $4$ that vanish when restricted to a fixed $0$-dimensional closed subscheme $S$ of $\ptwo$ of length $3$ or, in other terms, such that the plane quartic curves defined by the forms contain $S$: this condition forces the curve to be singular at the support of $S$.
\begin{rmk}
If $\varphi$ was surjective, then $\Fcal:=\ker(\varphi)$ would have been a locally free subsheaf of $\Ocal_{\wt{H}\times\ptwo}\otimes W_4$. Therefore, we would have obtained that there is a projective subbundle $\Pro(\Fcal)$ of $\wt{H}\times\Pro(W_4)$ whose projection on $\Pro(W_4)$ is injective over the locally closed subscheme parametrizing trinodal irreducible quartics. 

This would have enabled us to compute the relations in the equivariant Chow ring of $\Pro(W_4)$ that comes from this locus, using an argument which resembles the one we used for the case of curves with two nodes.

Unfortunately the morphism $\varphi$ is not surjective everywhere: nevertheless, there is a way to fix this problem by performing an elementary modification of vector bundles.
\end{rmk}
\begin{defi}
	We say that a $0$-dimensional subscheme $S$ of $\ptwo$ is \textit{rectilinear} if there exists a line that contains $S$.
\end{defi}
Let $H_r$ be the subset of $H$ whose points correspond to rectilinear subschemes, and let $\wt{H}_r$ be the preimage of $H_r$ in $\wt{H}$. Then we have the following:
\begin{propos}\label{propos.extension}
	The set $\wt{H}_r$ is a Cartier divisor and there exists a vector subbundle $\Fcal \subset W_4\times\wt{H}$ over $\wt{H}$ such that:
	\begin{enumerate}
		\item There exists an isomorphism $\Fcal|_{\wt{H}\setminus \wt{H}_r}\simeq\ker(\varphi)|_{\wt{H}\setminus \wt{H}_r}$.\\
		\item Let $[S]\in \wt{H}_r$ and let $l$ be the unique (up to scalar) linear form contained in $I_S$. Then:
		\[ \Fcal([S]) = \left\{ l^2q \text{ where } q\in H^0(\PP^2,\Ocal(2)) \right\} \subset H^0(\PP^2, I_S^2(4)). \]
	\end{enumerate}
\end{propos}
For the sake of readability we opt to take this result for granted, and to proceed with the computations. A proof of Proposition \ref{propos.extension} can be found in Section \ref{sec.modifications}.

Consider now the subvariety $\Zcal:=\Pro(\Fcal)$ of $\wt{H}\times\Pro(W_4)$, so that we have the diagram
\begin{equation}\label{eq:Z}\xymatrix{
	& \Zcal \ar[dl]_{p} \ar[dr]^{q} & \\
	\Pro(W_4)& &\wt{H}}
\end{equation}
Clearly, the image of $p$ is contained in the closure of the trinodal locus. Moreover, if we restrict to the locally closed subscheme of irreducible trinodal quartics, the morphism $p$ becomes injective. Therefore, we have put ourselves in a situation that resembles the basic setup exploited for the computation of the relations coming from the stratum of irreducible binodal quartics.
\begin{lemma}\label{lemma.generators trinodal}
	We have:
	\begin{enumerate}
		\item The image of $p_*:A^*_{\GL_{3}}(\Zcal)\to A^*_{\GL_{3}}(\Pro(W_4))$ is generated as an ideal by the elements $\eta_i:=p_*q^*\xi_i$, where the cycles $\xi_i$ are the generators of $A^*_{\GL_{3}}(\wt{H})$ as an $A^*_{\GL_{3}}$-algebra.\\
		\item The image of $A_*^{\GL_3}(\ov{Z}_{\rm binod})\arr A_*^{\GL_3}(\PP(W_4))$ is contained in the sum of the ideals $(I_{\wt{Z}},\delta_{\{ 1, 3 \}})$ and $\im(p_*)$.
	\end{enumerate}
\end{lemma}
\begin{proof}
	Due to the fact that $\Zcal$ is a projective bundle over $\wt{H}$ and a projective subbundle of $\Pro(W_4)\times\wt{H}$, we have that $A^*_{\GL_{3}}(\Zcal)$ is generated as an $A^*_{\GL_{3}}(\wt{H})$-algebra by $p^*h$. Then from the compatibility formula combined with the projection formula we easily deduce (1).
	
	Point (2) follows from the fact that we already know (see Proposition \ref{propos.twonodes}) that the image of $A^{\GL_3}_*(\ov{Z}_{\rm binod})\arr A_*^{\GL_3}(\PP(W_4))$ is contained in $(I_{\wt{Z}},\delta_{\{ 1, 3 \}})$ plus the ideal generated by the cycles coming from $Z_{\rm trinod}$ (i.e. those that are not contained in $Z_\reducible$). Cycles of this type can all be lifted to cycles in $A_{\GL_3}^*(\Zcal)$ because $$\Zcal|_{Z_{\rm trinod}}\arr Z_{\rm trinod}$$ is an isomorphism. Indeed, points in this locally closed subscheme correspond to irreducible quartics whose singular subscheme has length $3$, as if this was not the case then the curve would be reducible. This concludes the proof of the Lemma.
\end{proof}
\begin{rmk}\label{rmk.generators of the generic locus}
	We only have to care of those generators of $A^*_{\GL_{3}}(\wt{H})$ that are not entirely supported on the exceptional divisor $E$.
	
	Indeed, let $V\subset E$ be an subvariety: for every point $\Spec(k(v))$ of $V$, the corresponding quartic over $\Spec(k(v))$ must have a triple point. It follows that either it has only one singular point, or it is reducible, because the unique line passing through the triple point and another singular point would intersect the quartic in at least five points (counted with multiplicity). Therefore $p_*q^*[V]$ is not supported on the locally closed subscheme $Z_{\rm trinod}$ of irreducible quartics with exactly three nodes.
\end{rmk}
The last Remark in particular tells us that, among the generators of $A^*_{\GL_{3}}(\wt{H})$ as $A^*_{\GL_{3}}$-algebra, we can restrict our attention to those that come from $H$.

Let $T\subset \GL_{3}$ be the maximal subtorus of diagonal matrices.
Consider the diagram
$$\xymatrix{
	& \Sigma \ar[dl]_{\pr_1} \ar[dr]^{\pr_2} & \\
	H & & \ptwo }$$
where $\Sigma$ is the universal subscheme, and define $\Ecal_i:=\pr_{1*}\pr_2^*\Ocal(i)$. Then we have:
\begin{theorem}\cite[Theorem 1.1]{ES93}\label{theorem.generators hilbert}
	The sheaves $\Ecal_i$, for $0\leq i\leq 2$, are locally free of rank $3$, and $A^*(H)$ is generated as a $\Z$-algebra by the Chern classes of $\Ecal_i$. In particular $A^*(H)$ is generated as a $\Z$-module by the monomials in $c_j(\Ecal_i)$ of degree less or equal to $6$, where $\deg(c_j(\Ecal_i))=j$.
\end{theorem}
A useful Corollary is the following:
\begin{lemma}\label{lemma.formality}
	The equivariant Chow ring $A^*_{T}(H)$ is generated as an $A^*_{T}$-module by the monomials in $c_j^T(\Ecal_i)$ of degree less than or equal to $6$. The same statement holds if we replace $T$ with $\GL_3$.
\end{lemma}
\begin{proof}
	The first part is a consequence of the Hilbert scheme of points in $\PP^2$ being $T$-equivariantly formal. The second part follows directly from the identification $A^*_{\GL_3}(H)=A^*_T(H)^{S_3}$.
\end{proof}
Let $p$ and $q$ be the morphisms appearing in (\ref{eq:Z}) and let $\sigma:\wt{H}\arr H$ be the blow-down morphism. 
Lemmas \ref{lemma.generators trinodal} and \ref{lemma.formality}, combined with Remark \ref{rmk.generators of the generic locus} and Theorem \ref{theorem.generators hilbert}, readily imply that the set of cycles $p_*q^*\sigma^*\xi$, where $\xi$ is a polynomial of degree less or equal to $6$ in the algebra $\Z[c_j^{\GL_3}(\Ecal_i)]$, generate $\im(p_*)$ modulo the ideal $(I_{\wt{Z}},\delta_{\{ 1, 3 \}})$ (with a little abuse of notation, we use the same name for $\Ecal_i$ and its pullback along the blow-down morphism $\wt{H}\arr H$).

Actually, more is true. Define $\eta(i_1,\dots,i_7)$ as
\[ p_*q^*[c_1^{\GL_3}(\Ecal_0)^{i_1} c_2^{\GL_3}(\Ecal_0)^{i_2} c_1^{\GL_3}(\Ecal_1)^{i_3} c_2^{\GL_3}(\Ecal_2)^{i_4} c_3^{\GL_3}(\Ecal_1)^{i_5} c_2^{\GL_3}(\Ecal_2)^{i_6} c_3^{\GL_3}(\Ecal_2)^{i_7}]. \] 
Then we have the following Lemma:
\begin{lemma}\label{lemma.generators of im(p_*)}
	The ideal $\im(p_*)$ is generated modulo cycles in $(I_{\wt{Z}},\delta_{\{1,3\}})$ by the cycles $\eta(i_1,\dots,i_7)$ for $i_j \geq 0$ and $0\leq i_1+2i_2+i_3+2i_4+3i_5+2i_6+3i_7 \leq 6$.
\end{lemma}
\begin{proof}
	First observe that there is an injective morphism of vector bundles $\Ocal\hookrightarrow\Ecal_0$, where $\Ocal$ is the equivariantly trivial vector bundle: this implies that $c_3^{\GL_3}(\Ecal_0)=0$.
	
	Moreover, from \cite[pg. 345]{ES87}, we know that $2c_1(\Ecal_1)-c_1(\Ecal_2)=c_1(\Ecal_0)$.
	From Remark \ref{rmk.generators of the generic locus} and Theorem \ref{theorem.generators hilbert} we deduce the Lemma.
\end{proof}
We have reduced ourselves to the computation of the cycles $\eta(\underline{i})$, whose explicit expression in terms of the elements $c_i$ and $h_4$ can be obtained applying the localization formula. The final result is the following:
\begin{propos}\label{propos.class trinodal}
	We have $\eta(\underline{i})\in I_{\wt{Z}}$ for $i_j \geq 0$ and $0\leq i_1+2i_2+i_3+2i_4+3i_5+2i_6+3i_7 \leq 6$.
\end{propos}
In order to apply the localization formula, we need some preliminary computations. First we find all the points of $\wt{H}$ that are fixed by the action of $T\subset \GL_{3}$:
\begin{lemma}\label{lemma.fixed points of Htilde}
	The fixed locus of $\wt{H}$ consists of $31$ fixed points.
\end{lemma}
\begin{proof}
	Observe that the fixed locus of $H$ consists of $22$ fixed points $p_1,\dots,p_{22}$ corresponding to the subschemes with homogeneous ideals:
	\begin{align*}
	I_{p_1}&=(XY,XZ,YZ) & I_{p_7}&=(Z^2,XY,ZY) & I_{p_{13}}&=(X^2Y,Z)& I_{p_{19}}&=(X^3,Z) \\
	I_{p_2}&=(X^2,YZ,XY)& I_{p_8}&=(X^2,XY,Y^2)& I_{p_{14}}&=(X,YZ^2)& I_{p_{20}}&=(X,Z^3) \\
	I_{p_3}&=(X^2,YZ,XZ)& I_{p_9}&=(X^2,XZ,Z^2)& I_{p_{15}}&=(Z,XY^2)& I_{p_{21}}&=(Y,Z^3) \\
	I_{p_4}&=(Y^2,XZ,XY)& I_{p_{10}}&=(Y^2,YZ,Z^2) & I_{p_{16}}&=(XZ^2,Y)& I_{p_{22}}&=(Y^3,Z) \\
	I_{p_5}&=(Y^2,XZ,YZ)& I_{p_{11}}&=(X,Y^2Z) & I_{p_{17}}&=(X^3,Y)& \\
	I_{p_6}&=(Z^2,XY,XZ)& I_{p_{12}}&=(X^2Z,Y) & I_{p_{18}}&=(X,Y^3) 
	\end{align*}
	The $31$ fixed points of $\wt{H}$ are given by the $19$ fixed points $p_1,\dots,p_7,p_{11},\dots,p_{22}$ plus other $12$ fixed points contained in the exceptional divisor which we will describe in the next lines. 
	
	Recall from Remark \ref{rmk.fiber of extension} that the normal bundle of the non-curvilinear locus of $H$ at a point can be identified with the space of homogeneous cubic forms in the local coordinates of the point. This suggests that there are exactly twelve fixed points in $\wt{H}$ lying  over the points $p_8$, $p_9$, $p_{10}$ of $H$, corresponding to the pairs $(p_i,f)$ where $f$ is a monomial of degree $3$ in the local coordinates of $\Sigma(p)$.
	
	More precisely, recall that:
	$$ TH_p=\Hom_{\Ocal_{\Sigma(p)}}(\Ical_{\Sigma(p)}/\Ical_{\Sigma(p)}^2,\Ocal_{\Sigma(p)}) $$
	Let $p$ be a point in the non-curvilinear locus $H_{nc}\subset H$ and let $u$, $v$ be local coordinates near the support of $\Sigma(p)$. An element $\varphi$ of $\Hom_{\Ocal_{\Sigma(p)}}(\Ical_{\Sigma(p)}/\Ical_{\Sigma(p)}^2,\Ocal_{\Sigma(p)})$ is completely determined by its values at $u^2$, $uv$, $v^2$. Moreover, in order to be a morphism of $\Ocal_{\Sigma(p)}$-modules, such values must be of the form $au+bv$ for some $(a,b)\in k^{\oplus 2}$. From this we deduce that:
	$$ TH_p=\left \langle u^{2\vee}\otimes u,u^{2\vee}\otimes v,uv^{\vee}\otimes u,uv^{\vee}\otimes v,v^{2\vee}\otimes u,v^{2\vee}\otimes v \right\rangle $$
	To each element $\varphi$ of $TH_p$ corresponds the first order deformation of the subscheme $\Sigma(p)$ locally defined by the ideal:
	$$ (u^2+\epsilon\varphi(u^2),uv+\epsilon\varphi(uv),v^2+\epsilon\varphi(v^2)) $$
	A basis for the vector subspace $T(H_{nc})_p \subset TH_p$ is given by:
	$$ \left\langle u^{2\vee}\otimes 2u + uv^\vee\otimes v, uv^\vee\otimes u + v^{2\vee}\otimes 2v \right\rangle. $$
	This is because an element in the subspace above corresponds to the deformation whose local ideal is: 
	$$ (u^2+2a\epsilon u,uv+\epsilon(bu+av),v^2+2b\epsilon v)=(u+a\epsilon ,v+b\epsilon )^2. $$
	Thus a base for the normal space at $p$ is given by (equivalence classes of) the elements:
	$$ \left\{ \overline{u^{2\vee}\otimes u},\overline{u^{2\vee}\otimes v},\overline{v^{2\vee}\otimes u},\overline{v^{2\vee}\otimes v} \right\} .$$
	As a side remark, observe that the cubic form associated to an element $\varphi$ in the normal space at $p$ is the degree $3$ polynomial that appears among the generators of $(u^2+\epsilon\varphi(u^2),uv+\epsilon\varphi(uv),v^2+\epsilon\varphi(v^2))^2$ .
	In the end, we deduce that the fixed points of $\wt{H}$ contained in the exceptional divisors are all of the form $(p_i,\varphi)$, where $8\leq i\leq 10$ and $\varphi$ is one of the elements of the basis above.
\end{proof}	
\begin{lemma}\label{lemma.fixed points of Z}
	The fixed points of $\Zcal$ are of the form $([C],p)$ where $C=\{f=0\}$ is a plane quartic, $f$ is a monomial, $p$ is a fixed point of $\wt{H}$ and either:
	\begin{enumerate}
		\item  $p$ is not contained in the exceptional divisor, is non-rectilinear and $f\in I_p^2$ (fixed point of type $(1)$).
		\item $p=(p_i,g)$ is contained in the exceptional divisor of $\wt{H}$ and $f\in (I_{p_i}^2,g)$ (fixed point of type $(2)$).
		\item $I_p$ contains a degree $1$ monomial $l$ and $f\in (l^2)$ (fixed point of type $(3)$).
	\end{enumerate} 
\end{lemma}
\begin{proof}
	It easily follows from Lemma \ref{lemma.fixed points of Htilde} and from the construction of $\Zcal$.
\end{proof}
Observe that we have
$$ c_{top}^T(T\Zcal_{([C],p)})=c_{top}^T(T\wt{H}_p)c_{top}^T(T\Pro(V_p)_{[C]}) $$
where $\Pro(V_p)\subset \Pro(W_4)$ is the projective subspace of quartics $[Q]$ such that $([Q],p)$ is in $\Zcal$. Recall also that:
$$ TH_p=\Hom_{\Ocal_{\Sigma(p)}}(\Ical_{\Sigma(p)}/\Ical_{\Sigma(p)}^2,\Ocal_{\Sigma(p)}). $$
There are basically four cases that we need to consider in order to compute the top Chern classes of the tangent spaces of the fixed points of $\wt{H}$.
\begin{enumerate}
	\item Suppose that $p$ is the unique fixed point whose associated subscheme is supported on three distinct points. Then we have:
	$$ TH_p=\left\langle \dfrac{X}{Z}^\vee,\dfrac{Y}{Z}^\vee,\dfrac{X}{Y}^\vee,\dfrac{Z}{Y}^\vee,\dfrac{Y}{X}^\vee,\dfrac{Z}{X}^\vee \right\rangle .$$
	From this immediately follows:
	$$ c_{top}(T\wt{H}_p)=(l_3-l_1)(l_3-l_2)(l_2-l_1)(l_2-l_3)(l_1-l_2)(l_1-l_3).$$ 
	\item Suppose that $\Sigma(p)$ is supported on two distinct points whose local coordinates are respectively $(u,v)$ and $(s,t)$. Suppose moreover that the ideal of $\Sigma(p)$ near the first point is equal to $(u^2,v)$. Then we have:
	$$ TH_p=\left\langle u^{2\vee}\otimes 1, u^{2\vee}\otimes u, v^\vee\otimes 1,v^\vee\otimes u, s^\vee, t^\vee \right\rangle. $$
	Therefore:
	\begin{align*}
	c_{top}^T(T\wt{H}_p)=2\chi_u^2(-\chi_v)(\chi_u-\chi_v)\chi_s\chi_t.
	\end{align*} 
	Applying this formula we get the top Chern class of the tangent space of $p_i$ for $2\leq i \leq 7$ and $11\leq i \leq 16$.
	\item Suppose that $\Sigma(p)$ is curvilinear and supported on one point, whose local coordinates are $(u,v)$. Then the ideal of $\Sigma(p)$ near the support is equal to $(u^3,v)$. From this we deduce:
	$$ TH_p=\left\langle u^{3\vee}\otimes 1,u^{3\vee}\otimes u,u^{3\vee}\otimes u^2,v^\vee\otimes 1,v^\vee\otimes u,v^\vee\otimes u^2 \right\rangle. $$
	Thus we get:
	$$ c_{top}^T(T\wt{H}_p)=-6\chi_u^3(-\chi_v)(\chi_u-\chi_v)(2\chi_u-\chi_v). $$
	This formula gives us the top Chern classes of the tangent spaces of $p_i$ for $17\leq i\leq 22$.
\end{enumerate} 
The fourth case left to compute is the top Chern class of the tangent space of the fixed points contained in the exceptional divisor of $\wt{H}$. To do this we need a description as a $T$-representation of the tangent space of a $T$-fixed point in an equivariant blow-up scheme.
	\begin{lemma}\label{lemma.tangent of blow-up}
		Let $X$ be a smooth $T$-scheme and $Y\subset X$ a $T$-invariant, regularly embedded subscheme. Let $x$ be a $T$-fixed point of $X$, contained in $Y$, and let $\{\varphi_1,\dots,\varphi_n\}$ be a base for $TX_x$ such that, with respect to the induced action of $T$ on $TX_x$, we have $t\cdot\varphi_i=\chi_i(t)\varphi_i$. Suppose moreover that $\varphi_1,\dots,\varphi_m$ is a base for $TY_x\subset TX_x$. Let $\wt{X}$ be the blow-up of $X$ at $Y$. Then:
		$$ T\wt{X}_{(x,[\varphi_j])}=\left\langle
		\varphi_1,\dots,\varphi_m,\varphi_{m+1}\otimes \varphi_j^\vee,\dots,\widehat{\varphi_j \otimes \varphi_j^\vee},\dots,\varphi_n\otimes\varphi_j^\vee \right\rangle. $$
		In particular: $$c_{top}^T(T\wt{X}_{(x,[\varphi_j])})=c_1^T(\chi_1)\cdots c_1^T(\chi_m)(c_1^T(\chi_{m+1})-c_1^T(\chi_j))\cdots (c_1^T(\chi_n)-c_1^T(\chi_j)). $$
	\end{lemma}	
The proof is pretty standard and we omit it. Let $p$ be a fixed point of $H$ such that $\Sigma(p)$ is non-curvilinear and supported on one point. Let $(u,v)$ be local coordinates near the support of $\Sigma(p)$. Then we have seen in the proof of Lemma \ref{lemma.fixed points of Htilde} that
$$ TH_p=\left\langle u^{2\vee}\otimes 2u + uv^\vee\otimes v, uv^\vee\otimes u + v^{2\vee}\otimes v,u^{2\vee}\otimes u,u^{2\vee}\otimes v,v^{2\vee}\otimes u,v^{2\vee}\otimes v \right\rangle. $$
We have also seen that there are exactly four fixed points of $\wt{H}$ lying over $p$: they are $(p,[u^{2\vee}\otimes u])$, $(p,[u^{2\vee}\otimes v])$, $(p,[v^{2\vee}\otimes u])$, $(p,[v^{2\vee}\otimes v])$. Using Lemma \ref{lemma.tangent of blow-up} we can easily compute the top Chern classes of the tangent spaces of these points.

The computation of $c^T_{top}(T\Pro(V_p)_{[C]})$, where $C$ is a plane quartic defined by a degree $4$ monomial $X^aY^bZ^c$, is divided in three cases:
\begin{enumerate}
	\item Suppose $p\in\wt{H}$ is a fixed point of type $(1)$ (see lemma \ref{lemma.fixed points of Z}), then:
	$$ T\Pro(V_p)_{[C]}=\left\langle \dfrac{X^iY^jZ^k}{X^aY^bZ^c} \text{ where }X^iY^jZ^k \in I_{\Sigma(p)}^2 \right\rangle. $$
	Therefore:
	$$ c_{top}^T(T\Pro(V_p)_{[C]})=\prod ((i-a)l_1+(j-b)l_2+(k-c)l_3) $$
	where the product is taken over the triples $(i,j,k)$ such that $i,j,k\geq 0$, $i+j+k=4$, $(i,j,k)\neq (a,b,c)$, $X^iY^jZ^k \in I_{\Sigma(p)}^2$.\\
	\item Suppose $p=(p_i,g)$ is a fixed point of type $(2)$, then:
	$$ T\Pro(V_p)_{[C]}=\left\langle \dfrac{X^iY^jZ^k}{X^aY^bZ^c} \text{ where }X^iY^jZ^k \in (I_{\Sigma(p_i)}^2,g) \right\rangle. $$
	Therefore:
	$$ c_{top}^T(T\Pro(V_p)_{[C]})=\prod ((i-a)l_1+(j-b)l_2+(k-c)l_3). $$
	where the product is taken over the triples $(i,j,k)$ such that $i,j,k\geq 0$, $i+j+k=4$, $(i,j,k)\neq (a,b,c)$, $X^iY^jZ^k \in (I_{\Sigma(p_i)}^2,g)$.\\
	\item Suppose $p$ is a fixed point of type $(3)$, and let $l$ be a degree $1$ monomial contained in $I_{\Sigma(p)}$, then:
	$$ T\Pro(V_p)_{[C]}=\left\langle \dfrac{X^iY^jZ^k}{X^aY^bZ^c} \text{ where }X^iY^jZ^k \in (l^2) \right\rangle. $$
	Therefore:
	$$ c_{top}^T(T\Pro(V_p)_{[C]})=\prod ((i-a)l_1+(j-b)l_2+(k-c)l_3) $$
	where the product is taken over the triples $(i,j,k)$ such that $i,j,k\geq 0$, $i+j+k=4$, $(i,j,k)\neq (a,b,c)$, $X^iY^jZ^k \in (l^2)$.
\end{enumerate}
Recall that in the proof of Proposition \ref{propos.twonodes} we already computed the classes of the fixed points of $\Pro(W_4)$. Therefore, we are in position to prove Proposition \ref{propos.class trinodal}.
\begin{proof}[Proof of Proposition \ref{propos.class trinodal}]
	In order to apply the localization formula we need to compute the equivariant Chern roots of the restriction of $\Ecal_0$, $\Ecal_1$ and $\Ecal_2$ to the fixed points of $\wt{H}$. It is actually enough to perform these computations for the fixed points of $H$: given a fixed point $p$ in the exceptional divisor of $\wt{H}$, the Chern roots of the restriction of $\Ecal_i$ to this point are equal to the Chern roots of $\Ecal_i$ restricted to the image of $p$ via the blow-down morphism.
	
	In the following, we denote $S_i$ the subscheme parametrized by the point $p_i$ in $H$, where $p_i$ is as in the proof of Lemma \ref{lemma.fixed points of Htilde}. We start by computing the Chern roots of $\Ecal_0(p_i)$.
	
	We have that $\Ecal_0(p_1)$ is isomorphic $T$-equivariantly to the direct sum of the residue fields of the three points in the support of $S_1$, therefore the Chern roots of $\Ecal_0(p_1)$ are all equal to zero. 
	
	The $T$-representation $\Ecal_0(p_2)$ is isomorphic to the direct sum of the residue fields of the two points in the support of $S_2$ plus the representation generated by the rational function $\frac{X}{Z}$. Therefore the only non-zero Chern root of the $T$-representation $\Ecal_0(p_2)$ is equal to $l_1-l_3$. The Chern roots of $\Ecal_0(p_i)$ for $i=3,4,5,6,7$ are obtained in the same way, after having appropriately exchanged the roles of $X$, $Y$ and $Z$.
	
	The $T$-representation $\Ecal_0(p_8)$ is isomorphic to the direct sum of the residue field of the unique point in the support of $S_8$ plus the representation generated by the rational functions $\frac{X}{Z}$ and $\frac{Y}{Z}$. This implies that the non-zero Chern roots of the $T$-representation $\Ecal_0(p_8)$ are $l_1-l_3$ and $l_2-l_3$. The Chern roots of $\Ecal_0(p_i)$ for $i=9,10$ are obtained in the same way, after having appropriately exchanged the roles of $X$, $Y$ and $Z$.
	
	The $T$-representation $\Ecal_0(p_{11})$ is isomorphic to the direct sum of the residue fields of the points in the support of $S_{11}$ plus the representation generated by the rational function $\frac{Y}{Z}$. Therefore, the non-zero Chern root of $\Ecal_0(p_{11})$ is equal to $l_2-l_3$. The Chern roots of $\Ecal_0(p_i)$ for $i=12,\cdots,16$ are obtained in the same way, after having appropriately exchanged the roles of $X$, $Y$ and $Z$.
	
	The $T$-representation $\Ecal_0(p_{17})$ is isomorphic to the direct sum of the residue field of the unique point in the support of $S_{17}$ plus the representation generated by the rational functions $\frac{X}{Z}$ and $\frac{X^2}{Z^2}$. Therefore, the non-zero Chern roots of $\Ecal_0(p_{17})$ are equal to $l_1-l_2$ and $2(l_1-l_2)$. The Chern roots of $\Ecal_0(p_i)$ for $i=18,\cdots,22$ are obtained in the same way, after having appropriately exchanged the roles of $X$, $Y$ and $Z$.
	
	We now compute the Chern roots of $\Ecal_1(p_1)$.
	
	The representation $\Ecal_1(p_1)$ is generated by the restriction to this point of the linear forms $X$, $Y$ and $Z$. Therefore, the Chern roots are $l_1$, $l_2$ and $l_3$.
	
	The linear forms $X$, $Y$ and $Z$ also form a basis for $\Ecal_1(p_i)$, where $i=2,\ldots,10$. We deduce that the Chern roots of all these representations are $l_1$, $l_2$ and $l_3$.
	
	Consider now the subscheme $S_{11}$: a basis of $H^0(S_{11},\Ocal(1))$ is given by the restrictions of the global sections $Y$ and $Z$ plus the restriction of the local section $y\cdot Z$, which is defined in the $Z\neq 0$ chart of $\PP^2$. We deduce that the Chern roots of $\Ecal_1(p_{11})$ are equal to $l_2$, $l_3$ and $l_2$. The Chern roots of $\Ecal_1(p_i)$ for $i=12,\dots,16$ are obtained in a similar fashion, after permuting $X$, $Y$ and $Z$ accordingly.
	
	The subscheme $S_{17}$ is all contained in the $Z\neq 0$ chart of $\PP^2$: looking at the structure sheaf of $S_{17}$, we deduce that $H^0(S_{17},\Ocal(1))$ is generated by $Z$, $x\cdot Z$ and $x^2\cdot Z$: this implies that the Chern roots of $\Ecal_1(p_{17})$ are $l_3$, $l_1$ and $2l_1-l_3$. Again, the Chern roots of $\Ecal_1(p_i)$ for $i=18,\dots,22$ can be deduced with a similar argument, after having exchanged the roles of $X$, $Y$ and $Z$.
	
	We now compute the Chern roots of $\Ecal_2(p_i)$.	
	It is easy to see that the representation $\Ecal_2(p_1)$ is isomorphic to the representation of quadratic forms that do not vanish on all the three points in the support of $S_1$: this representation is generated by $X^2$, $Y^2$ and $Z^2$. Therefore, the Chern roots of $\Ecal_2(p_1)$ are $2l_1$, $2l_2$ and $2l_3$.
	
	With a similar argument, we see that the generators of the representation $\Ecal_2(p_2)$ are equal to $Y^2$, $XZ$ and $Z^2$. Therefore, the Chern roots of $\Ecal_2(p_2)$ are equal to $2l_2$, $l_1+l_3$ and $2l_3$. The Chern roots of $\Ecal_2(p_i)$ for $i=3,4,5,6,7$ are obtained exchanging the roles of $X$, $Y$ and $Z$ in an appropriate way.
	
	The generators of the representation $\Ecal_2(p_8)$ are equal to $XZ$, $YZ$ and $Z^2$. Therefore, the Chern roots of $\Ecal_2(p_8)$ are equal to $l_1+l_3$, $l_2+l_3$ and $2l_3$. The Chern roots of $\Ecal_2(p_i)$ for $i=9,10$ are obtained exchanging the roles of $X$, $Y$ and $Z$ in an appropriate way.
	
	The generators of the representation $\Ecal_2(p_{11})$ are equal to $Y^2$, $YZ$ and $Z^2$. Therefore, the Chern roots of $\Ecal_2(p_{11})$ are equal to $2l_2$, $l_2+l_3$ and $2l_3$. The Chern roots of $\Ecal_2(p_i)$ for $i=12,\cdots,16$ are obtained exchanging the roles of $X$, $Y$ and $Z$ in an appropriate way.
	
	Finally, the generators of the representation $\Ecal_2(p_{17})$ are equal to $Y^2$, $YZ$ and $Z^2$. Therefore, the Chern roots of $\Ecal_2(p_{17})$ are equal to $2l_2$, $l_2+l_3$ and $2l_3$. The Chern roots of $\Ecal_2(p_i)$ for $i=18,\cdots,22$ are obtained exchanging the roles of $X$, $Y$ and $Z$ in an appropriate way.
	
	We now have all the ingredients necessary to write down the explicit localization formula for the cycles we are interested in. More precisely, we have:
	\[ \eta(\underline{i})=\sum_{(C,p)} \frac{\text{polynomial in the equivariant Chern classes of } \Ecal_j(p)}{c_{top}^T(TZ_{C,p})}\cdot [C] \]
	where the sum is taken over all the fixed points $(C,p)$ contained in $\Zcal\subset \PP(W_4)\times \wt{H}$. A straightforward computation with Mathematica concludes the proof.	
\end{proof}
\begin{corol}\label{corol.image in ideal}
	The image of $A^{\GL_3}_*(\ov{Z}_{\rm binod})\arr A^{\GL_3}_*(\PP(W_4))$ is contained in $(I_{\wt{Z}},\delta_{\{ 1, 3 \}})$.
\end{corol}
\begin{proof}
	By Lemma \ref{lemma.generators trinodal} we know that the image of $A^{\GL_3}_*(\ov{Z}_{\rm binod})\arr A^{\GL_3}_*(\PP(W_4))$ is contained in $(I_{\wt{Z}},\delta_{\{ 1, 3 \}})$ plus $\im(p_*)$. By Proposition \ref{propos.class trinodal}, combined with Lemma \ref{lemma.generators of im(p_*)}, the latter ideal is contained in $(I_{\wt{Z}},\delta_{\{ 1, 3 \}})$. This concludes the proof.
\end{proof}

\section{Geometry of ${\rm Hilb}^3\PP^2$}\label{sec.modifications}

This Section is devoted to prove some results on $H:={\rm Hilb}^3\PP^2$ that were used in the previous Section.
\subsection{Extension of $\Sigma_1$ to $\wt{H}$}
As before, we denote $\Sigma$ the universal family over $H$ and $\Sigma_1$ its first infinitesimal neighborhood. 

Recall that $Y$ is the closed subvariety of $H$ whose points $[S]$ correspond to $0$-dimensional subschemes $S\subset \PP^2$ of length $3$, supported on one point and non-curvilinear. Observe that $Y\simeq\PP^2$.

Let $\wt{H}$ be the blow-up of $H$ along $Y$ and let us denote the pullback of $\Sigma_1$ to $\wt{H}$ with the same name. The goal of this Subsection is to prove Proposition \ref{propos.Htilde}.

Observe that $\Sigma_1$ induces a rational morphism $\wt{H}\dashrightarrow\hilb^9\ptwo$ which is defined away from the exceptional divisor $E$ in $\wt{H}$. Let $\Gamma$ be the closure of the graph of this rational morphism inside $\wt{H}\times\hilb^9\ptwo$, so that we have
$$\xymatrix{
	\Gamma \subset \wt{H}\times\hilb^9\ptwo \ar[r]^{\hspace{0.5cm}q} \ar[d]^{p} & \hilb^9\ptwo \\
	\wt{H} & }$$
Observe that over $\Gamma$ there is a well defined flat family of $0$-dimensional closed subscheme of $\ptwo$ of length $9$, given by the restriction to $\Gamma$ of $q^*\Upsilon$, where $\Upsilon$ is the universal family over $\hilb^9\ptwo$.
\begin{propos}\label{propos.isomorphism}
	The projection $\Gamma\arr \wt{H}$ is an isomorphism.
\end{propos}
To prove the Proposition above, we need a technical result. Consider the following situation: let $R$ be a complete DVR with uniformizer $t$ and let $K$ (resp. $k$) be the fraction field (resp. the residue field) of $R$. Fix a point $[S]$ in $Y\subset H$ which corresponds to a $0$-dimensional closed subscheme $S\subset \ptwo$ of length $3$, supported on only one closed point, non curvilinear. Let $\gamma:\Spec(R)\arr H$ be a morphism such that:
\begin{enumerate}
	\item [(i)] the induced morphism $\gamma_K:\Spec(K)\arr H$ factorizes through the open subscheme $U\subset H$ whose points correspond to triplets of distinct closed points in $\ptwo$;
	\item [(ii)] $\gamma(\Spec(k))=[S]$.
\end{enumerate}
Call $\nu(\gamma)$ the induced tangent vector $\Spec(k[\epsilon])\arr H$ and let $[\nu(\gamma)]$ be the image of $\nu(\gamma)$ through the quotient morphism $TH_{[S]}\arr N_{[S]}$, where $TH_{[S]}$ (resp. $N_{[S]}$) denotes the tangent vector space of $H$ (resp. the normal vector space of $Y$ in $H$) at $[S]$.

Given a morphism $\gamma$ as above, we can define a lifting $\wt{\gamma}:\Spec(R)\arr \hilb^9\ptwo$ as follows: let $\wt{\gamma}_K$ be the composition of $\gamma_K$ with the morphism $U\arr \hilb^9\ptwo$ induced by $\Sigma_1|_U$. By the properness of $\hilb^9\ptwo$ there exists a unique extension of $\wt{\gamma}_K$ to a morphism $\wt{\gamma}:\Spec(R)\to \hilb^9\ptwo$. We are ready to state the technical Lemma:
\begin{lemma}\label{lemma.finiteness of the fiber}
	Fix a point $[S]$ in $Y$ and let $\gamma_1$ and $\gamma_2$ be as above. Suppose moreover that $[\nu(\gamma_1)]$ and $[\nu(\gamma_2)]$ are non-zero and linearly dependent. Then $\wt{\gamma_1}(\Spec(k))=\wt{\gamma_2}(\Spec(k))$.
\end{lemma}
\begin{rmk}\label{rmk.normal directions are cubic forms}
	Before giving a complete proof of the Lemma, which is rather technical and involved, let us briefly sketch what is the geometric idea behind it: indeed, what we are going to show is that the normal vectors of $Y$ in $H$ at $[S]$ can be identified with binary forms of degree $3$. 
	
	Roughly, the reason for this to be true is that a deformation of $S$ along a generic normal direction corresponds to a family of three distinct points approaching the support of $S$ in a non-curvilinear way. In turn, three distinct points determine three distinct lines. When the points degenerate to the support of $S$, the lines collapse to three lines passing through the support of $S$. The product of the equations defining these three lines is exactly the binary form of degree $3$ that we identify with the normal direction.
	
	This binary form $f$ has also the property that, if we take a degeneration of the three doubled points along any path and we denote $\wt{S}$ the limiting $0$-dimensional, length $9$ subscheme of $\ptwo$, then $I_{\wt{S}}=(f)+(x,y)^4$, i.e. the limit $\wt{S}$ depends not on the whole degeneration but only on the normal direction determined by the degeneration, which is exactly the content of Lemma \ref{lemma.finiteness of the fiber}.	
\end{rmk}
\begin{proof}[Proof of Lemma \ref{lemma.finiteness of the fiber}]
	Let $[\wt{S}_i]:=\wt{\gamma_i}(\Spec(k))$, so that $\wt{S}_i$ for $i=1,2$ are $0$-dimensional closed subschemes of $\ptwo$ of length $9$, supported on only one point. We want to show that $\wt{S}_1=\wt{S}_2$. Without loss of generality we can suppose that these schemes are supported on the origin in $\mathbb{A}^2$. It is easy to see that, if  $I_S=(x,y)^2$, then $I_{\wt{S}_i}\supsetneq (x,y)^4$ for $i=1,2$. By hypotheses we know that $\wt{S}_i$ has length $9$, thus if we find any other element in $I_{\wt{S}_i}$ which is not in $(x,y)^4$ we completely determine the ideal $I_{\wt{S}_i}$, and therefore the scheme $\wt{S}_i$.
	
	Observe that over $U\subset H$ there is a degree $6$ \'{e}tale cover $V$, which is the open subscheme of $\ptwo\times\ptwo\times\ptwo$ of triplets made of distinct points. By pulling back $V$ along $\gamma_{K,i}$ we obtain a degree $6$ \'{e}tale cover $V_K$ of $\Spec(K)$: by choosing a section of $V_K\arr \Spec(K)$, we obtain a morphism $\Spec(K)\arr V$. In other terms, we can distinguish the three points $p^i(t)$, $q^i(t)$ and $r^i(t)$ of $\ptwo_K$ whose union corresponds to the $0$-dimensional, length $3$ closed subscheme $\gamma_{K,i}^*\Sigma\subset \ptwo_K$.
	By hypotheses, we can write:
	\begin{align*}
	p^i(t)&=[tp^i_1+t^2(\cdots):tp^i_2+t^2(\cdots):1],\\
	q^i(t)&=[tq^i_1+t^2(\cdots):tq^i_2+t^2(\cdots):1],\\
	r^i(t)&=[tr^i_1+t^2(\cdots):tr^i_2+t^2(\cdots):1].
	\end{align*}
	Let $l_{p^iq^i}(t)$ be an equation for the unique line passing through the points $p^i(t)$ and $q^i(t)$, and define in a similar way the polynomials $l_{p^ir^i}$ and $l_{q^ir^i}$. We have:
	\begin{align*}
	l_{p^iq^i}(t)&=t[(p^i_2-q^i_2)x-(p^i_1-q^i_1)y]+t^2[\cdots],\\
	l_{p^ir^i}(t)&=t[(p^i_2-r^i_2)x-(p^i_1-r^i_1)y]+t^2[\cdots], \\
	l_{q^ir^i}(t)&=t[(q^i_2-r^i_2)x-(q^i_1-r^i_1)y]+t^2[\cdots]. 
	\end{align*}
	Then the cubic defined by the equation $f^i(t)=t^{-3}l_{p^iq^i}(t)\cdot l_{p^ir^i}(t)\cdot l_{q^ir^i}(t)$ has three nodes in the three points, thus $f^i(t)$ is in the ideal of $\gamma_{K,i}^*\Sigma_1=\wt{\gamma}_{K,i}^*\Upsilon$, where $\Upsilon$ is the universal family over $\hilb^9\ptwo$. Therefore, the polynomial $f^i(0)$ is contained in $I_{\wt{S}_i}$ and by construction it is not in $(x,y)^4$. Therefore, to prove the Lemma is enough to show that $f^1(0)=f^2(0)$.

	An element $v$ of $TH_{[S]}$ may be identified with a triplet of linear forms in two variables $(L,M,N)$, where the ideal of $v^*\Sigma$ is equal to $(x^2+\eps L,xy+\eps M,y^2+\eps N)$. An element $v$ of $TY_{[S]}$ can then be seen as a pair $(a,b)$, where the ideal of $v^*\Sigma$ is equal to $((x+\eps a)^2,(x+\eps a)(y+\eps b),(y+\eps b)^2)$. This implies that two tangent vectors $v=(L,M,N)$ and $v'=(L',M',N')$ have the same image in $N_{[S]}$ if and only if there exists a pair $(a,b)$ such that $v'^*\Sigma=((x+\eps a)^2+\eps L,(x+ \eps a)(y+\eps b)+\eps M,(y+\eps b)^2+\eps N)$.
	
	Let $(x^2+L^i(t),xy+M^i(t),y^2+N^i(t))$ be the ideal of $\gamma_i^*\Sigma$, where:
	\begin{align*}
	L^i(t)&=tL^i_1+t^2(\cdots),\\
	M^i(t)&=tM^i_1+t^2(\cdots),\\
	N^i(t)&=tN^i_1+t^2(\cdots).
	\end{align*}
	Then the condition that $\nu(\gamma_1)=\nu(\gamma_2)$ implies that the ideal of $\gamma_2^*\Sigma$ is generated by
	\begin{align*}
	&(x+ta+t^2(\cdots))^2+tL^1_1+t^2(\cdots),\\
	&(x+ta+t^2(\cdots))(y+tb+t^2(\cdots))+tM^1_1+t^2(\cdots),\\
	&(y+tb+t^2(\cdots))^2+tN^1_1+t^2(\cdots).
	\end{align*}
	This readily implies that:
	\begin{align*}
	p^2_j(t)&=t(p^1_j+s_j)+t^2(\cdots),\\
	q^2_j(t)&=t(q^1_j+s_j)+t^2(\cdots),\\
	r^2_j(t)&=t(r^1_j+s_j)+t^2(\cdots).
	\end{align*}
	Plugging these expressions for $p^2(t)$, $q^2(t)$ and $r^2(t)$ in $f^2(t)$ immediately shows that $f^1(0)=f^2(0)$. This concludes the proof of the Lemma.
\end{proof}

The technical Lemma \ref{lemma.finiteness of the fiber} allows us to prove Proposition \ref{propos.isomorphism}.
\begin{proof}[Proof of Proposition \ref{propos.isomorphism}]
	It is enough to show that $\Gamma\to\wt{H}$ is birational and finite: this allows us to conclude by normality of $\wt{H}$. The only thing we need to prove is finiteness. Actually, we already know that $\Gamma\to \wt{H}$ is proper, so that is enough to check quasi-finiteness, that has to be shown only for the fibers over the points in the exceptional divisor of $\wt{H}$.
	
	Suppose that there exists a point $([S],\nu)$ in the exceptional divisor such that the fiber has positive dimension. Using the same notation as in the proof of lemma \ref{lemma.finiteness of the fiber}, this means that we can find two morphisms $\gamma_1,\gamma_2:\Spec(R)\to H$ such that $\gamma_i(\Spec(k))=[S]$, the induced normal vectors $[\nu(\gamma_i)]=v$ and $\wt{\gamma}_1(\Spec(k))\neq\wt{\gamma}_2(\Spec(k))$. By Lemma \ref{lemma.finiteness of the fiber}, this is not possible.
\end{proof}
Proposition \ref{propos.Htilde} is an immediate corollary to Proposition \ref{propos.isomorphism}.
\begin{rmk}
	The same arguments used to prove Lemma \ref{lemma.finiteness of the fiber} and Proposition \ref{propos.isomorphism} can be used to prove that there is an isomorphism between $\wt{H}$ and the blow-up of $\hilb^3\ptwod$ along the same locus, and such isomorphism does not depend on the choice of an identification $\ptwo\simeq\ptwod$. The isomorphism $\wt{H}\simeq \wt{\hilb}^3\ptwod$ sends the exceptional divisor of $\wt{H}$ to the closed subscheme whose points correspond to $0$-dimensional, length $3$ closed subscheme of $\ptwod$ that are contained in a line.
\end{rmk}
\subsection{Construction of $\Zcal$}
Consider the diagram:
$$\xymatrix{
	\wt{H}\times\ptwo \ar[r]^{\hspace{0.3cm}\pr_2} \ar[d]^{\pr_1} & \ptwo \\
	\wt{H} }$$
Recall that in Section \ref{sec.irr} we considered a morphism of locally free sheaves
$$\varphi:\pr_{1*}\pr_2^*\Ocal(4) \arr \pr_{1*}(\Ocal_{\wt{\Sigma}_1}\otimes\pr_2^*\Ocal(4)).$$
The fiber over a fixed point in $\wt{H}$ of the kernel of $\varphi$ can be interpreted as the space of ternary forms of degree $4$ that vanish when restricted to a fixed $0$-dimensional closed subscheme $S$ of $\ptwo$ of length $9$ or, in other terms, such that the plane quartic curves defined by the forms contain $S$: this condition forces the curve to be singular at the support of $S$.

The morphism $\varphi$ is not surjective everywhere. In this Subsection we fix this problem by performing an elementary modification of vector bundles, so to construct a variety $\Zcal$ birational to $\ker(\varphi)$ which is a projective subbundle of $\wt{H}\times\PP(W_4)$. The final result is a proof of Proposition \ref{propos.extension}.
\begin{defi}
	We say that a $0$-dimensional subscheme $S$ of $\ptwo$ is \textit{rectilinear} if there exists a line that contains $S$.
\end{defi}
Let $H_r$ be the subset of $H$ whose points correspond to rectilinear subschemes. We can put on $H_r$ the structure of a regularly embedded closed subscheme as follows: consider the diagram:
$$ \xymatrix{
	 H\times\ptwo \ar[r]^{p} \ar[d]^{q} & \ptwo \\
	 H}$$
Then we have the morphism of locally free sheaves:
$$ \psi: q_*p^*\Ocal(1) \arr q_*(\Ocal_{\Sigma}\otimes p^*\Ocal(1)). $$
It is easy to see that $\psi$ is generically surjective. Moreover, the closed subscheme where the rank of $\psi$ is not maximal coincides with $H_r$, that in this way inherits a scheme structure. Observe that $H_r$ is defined as the vanishing locus of the determinant of $\psi$, thus it is a Cartier divisor.

Let us denote $\wt{H}_r$ the pullback of $H_r$ to $\wt{H}$. Then we have:
\begin{lemma}\label{lemma.surjectivity}
	The set of points of $\wt{H}$ where the morphism $$\varphi:\pr_{1*}\pr_2^*\Ocal(4) \arr \pr_{1*}(\Ocal_{\wt{\Sigma}_1}\otimes\pr_2^*\Ocal(4))$$ is not surjective coincides with the set of points of $\wt{H}_r$. In particular, the closed subscheme $D_1(\varphi)$ has codimension $1$ in $\wt{H}$.
\end{lemma}
\begin{proof}
	Fix a closed point $x$ in $\wt{H}\setminus\wt{H}_r$: we want to show that $\varphi(x)$ is surjective. Suppose that $x$ corresponds to a subscheme of $\ptwo$ supported on three distinct points $p$, $q$ and $r$, that by hypothesis are not contained in a line.
	
	Choose an equation $l_0$ for the unique line that passes through $q$ and $r$. Then choose three linear forms $l_i$ for $i=1,2,3$ such that the line of equation $l_1=0$ does not contain $p$ and the lines associated to $l_2$, $l_3$ pass through $p$ with linearly independent tangent directions, and define $f_i:=l_0^3l_i$.
	
	Repeat now this process exchanging the roles of $p$, $q$ and $r$: in this way we obtain nine quartics $f_1,\dots,f_9$ such that their images in $\pr_{1*}(\Ocal_{\wt{\Sigma}_1}\otimes\pr_2^*\Ocal(4))(x)$ are linearly independent by construction. This shows that $\varphi(x)$ is surjective.
	
	An opportune modification of the argument above shows also that $\varphi(x)$ is surjective when the subscheme associated to $x$ is non-rectilinear and supported on two points. The surjectivity of $\varphi$ over the exceptional divisor of $\wt{H}$ is obvious.
	
	Now let $x$ be a point in $H_r$. Let $l$ be a linear form defining the unique line that contains $S$, the subscheme of $\ptwo$ associated to the point $x$ in $\wt{H}$. Let $V\subset H^0(\ptwo,\Ocal(3))$ be the vector subspace of ternary forms of degree $3$ that vanish on $S$: then $V$ has dimension $7$ and multiplication by $l$ induces an inclusion $V\hookrightarrow \ker(\varphi(x))$. This proves that $\ker(\varphi(x))$ has dimension strictly greater than $6$, thus $\varphi(x)$ cannot be surjective.
\end{proof}
The Lemma above tells us that $D_1(\varphi)$ and $\wt{H}_r$ have the same points, but for the moment we do not know anything about the scheme structure of $D_1(\varphi)$. 
\begin{lemma}\label{lemma.D_1(phi) is CM}
	The closed subscheme $D_1(\varphi)$ is a Cartier divisor.
\end{lemma}
The proof of the Lemma above will be in several steps. Our strategy will consist in finding a locally free subsheaf $\Vcal$ of $\pr_{1*}\pr_2^*\Ocal(4)$ of rank $6$ such that $\Vcal\subset\ker(\varphi)$. This would imply that $D_1(\varphi)$ is schematically defined as the vanishing locus of the determinant of $\overline{\varphi}$, where
$ \overline{\varphi}:\pr_{1*}\pr_2^*\Ocal(4)/\Vcal \arr \pr_{1*}(\Ocal_{\wt{\Sigma}_1}\otimes\pr_2^*\Ocal(4)) $, 
and therefore it is a Cartier divisor.

Consider the following situation:
\begin{itemize}
	\item[$(\star)$] 	Let $x$ be a closed point of $\wt{H}_r$. Let $R$ be a complete DVR, $K$ (resp. $k$) its fraction field (resp. residue field). Let $\gamma:\Spec(R)\arr \wt{H}$ be a morphism such that $\gamma(\Spec(k))=x$ and the induced morphism $\gamma_K:\Spec(K)\arr \wt{H}$ factorizes through the open subscheme $U$ parametrizing subschemes supported on distinct points. Let $\wt{\gamma}_K:\Spec(K)\arr\Gr(6,15)$ be the composition of $\gamma_K$ with the morphism $U\arr \Gr(6,15)$ induced by $\ker(\varphi)|_U$, and let $\wt{\gamma}:\Spec(R)\arr \Gr(6,14)$ be its unique extension.
\end{itemize}
Then we have:
\begin{lemma}\label{lemma.limiting vector space}
	Suppose to be in the situation ($\star$) and let $S$ be the closed subscheme of $\ptwo$ associated to $x$. Choose a linear form $l$ whose zero locus in $\ptwo$ is the unique line that contains $S$. Then $\wt{\gamma}(\Spec(k))$ is the closed point in $\Gr(6,15)$ associated to the vector subspace
	$$ V:=\left\{ l^2q\text{ }|\text{ }q\in H^0(\ptwo,\Ocal(2)) \right\} \subset H^0(\ptwo,\Ocal(4)). $$
\end{lemma}
\begin{proof}
	It is convenient to think of $\Gr(6,15)$ as the parameter space for $5$-dimensional projective subspaces of $|\Ocal(4)|=\Pro(W_4)$. If we show that there are six quartic curves in $\wt{\gamma}(\Spec(k))$ which are union of a double line and a double conic and which span a $\Pro^5$, we are done.
	
	Suppose first that $S$, the closed subscheme of $\ptwo$ associated to $x$, is supported on three distinct points. Arguing as in the proof of Lemma \ref{lemma.finiteness of the fiber}, we can actually distinguish three $K$-points $p(t)$, $q(t)$ and $r(t)$ of $\ptwo_K$ which are the support of $\gamma_K(\Spec(K))$. Observe that by hypothesis these three points are distinct, non-rectilinear and their limit consists of the three points $p$, $q$ and $r$ of $S$. 
	
	Consider the plane quartic curve $C_t$ in $\ptwo_K$ which is the union of the lines $\ov{p(t)q(t)}$ and $\ov{p(t)r(t)}$ and of a conic $Q_t$ passing through $q(t)$ and $r(t)$. By construction, this curve $C_t$ is in the subspace of $|\Ocal(4)|$ corresponding to $\wt{\gamma}_K(\Spec(K))$. Moreover, the limit curve $C_0$ of $C_t$ is the union of the double line $\ov{pq}=\ov{qr}$ and of a conic $Q_0$ that passes through $q$ and $r$. In this way we deduce that all the quartics which are the union of the double line $\ov{pq}$ and of a conic that passes through $q$ and $r$ are contained in the subspace of $|\Ocal(4)|$ associated to $\wt{\gamma}(\Spec(k))$.
	
	Exchanging the roles of $p(t)$, $q(t)$ and $r(t)$ in the reasoning above, we obtain that the subspace associated to $\wt{\gamma}(\Spec(k))$ contains all the quartic curves that are union of the double line $\ov{pq}$ and of a conic that passes through two of the three points $p$, $q$ and $r$. It is immediate to verify that this is a $5$-dimensional projective space, thus we actually completely determined the subspace associated to $\wt{\gamma}(\Spec(k))$.
	
	Suppose now that $S$ is supported on only two points $p$ and $q=r$. The argument that we used in the case before fails to produce a subspace of dimension $5$, because it only tells us that the limit subspace contains all the quartic curves that are union of the double line and of a conic that passes through $q$. Take then a double conic that contains $p(t)$, $q(t)$ and $r(t)$: its limit will be the union the double line $\ov{pq}$ and the double of another line that is free to move in $|\Ocal(1)|$. In this way we have completely determined the subspace associated to $\wt{\gamma}(\Spec(k))$ also in this case.
	
	The last case, i.e. when $S$ is supported on a single point, is handled with the same method.
\end{proof}
Observe that $\ker(\varphi)$, once restricted to $U:=\wt{H}\setminus\wt{H}_r$, is a rank $6$ locally free subsheaf of $\pr_{1*}\pr_2^*\Ocal(4)\simeq\Ocal_U^{\oplus 15}$. This induces a morphism $U\arr \Gr(6,15)$. Let $\Gamma$ be the closure of the graph of this morphism in $\wt{H}\times \Gr(6,15)$, so that we have:
$$\xymatrix{
	\Gamma \subset \wt{H}\times\Gr(6,15)\ar[r]^{\hspace{0.5cm}q} \ar[d]^{p} & \Gr(6,15) \\
	\wt{H} & }$$
\begin{lemma}\label{lemma.extension vector bundle}
	The projection $p:\Gamma\arr\wt{H}$ is an isomorphism.
\end{lemma}
\begin{proof}
	Arguing as in the proof of Proposition \ref{propos.isomorphism}, it is enough to show that the fibers of $p$ over the points in $\wt{H}_r$ are $0$-dimensional. 
	
	Fix a point $x$ in $\wt{H}_r$: then Lemma \ref{lemma.limiting vector space} combined with the Zariski Main Theorem tells us that the fiber over $x$ has a unique generic point which is closed, thus the fiber consists of only one closed point.	
\end{proof}
\begin{proof}[Proof of Lemma \ref{lemma.D_1(phi) is CM}]
The restriction to $\Gamma$ of the tautological sheaf of $\Gr(6,15)$ is by construction an extension of $\ker(\varphi)|_U$, that thanks to Lemma \ref{lemma.extension vector bundle} can be regarded as locally free sheaf of rank $6$ on $\wt{H}$. 

Let us call this sheaf $\Vcal$: by construction we have that $\Vcal$ is contained in $\ker(\varphi)$, and therefore $D_1(\varphi)=D_1(\ov{\varphi})$, where
$$ \ov{\varphi}:\pr_{1*}\pr_2^*\Ocal(4)/\Vcal \arr \pr_{1*}(\Ocal_{\wt{\Sigma}_1}\otimes\pr_2^*\Ocal(4)) $$
This means that $D_1(\varphi)$ is schematically defined as the vanishing locus of the determinant of $\varphi$, thus it is a Cartier divisor.
\end{proof}

Let $\Lcal$ be the cokernel of the morphism $\varphi$ restricted to $D_1(\varphi)$. By Grauert's Theorem, it is immediate to verify that $\Lcal$ is actually an invertible sheaf. Recall the following well known fact on elementary modifications of locally free sheaves (see for instance \cite{Mar82}):
\begin{lemma}\label{lemma.elementary modification}
	Let $X$ be smooth scheme, $i:D\hookrightarrow X$ a Cartier divisor and $\Ecal$ a locally free sheaf on $X$. Suppose to have a surjective morphism
	$ \Ecal|_D\arr \Lcal $
	of locally free sheaves on $D$, where $\Lcal$ is an invertible sheaf. Then the kernel of the induced morphism of sheaves
	$ \Ecal\arr i_*\Lcal $
	is locally free.
\end{lemma}
Thanks to Lemma \ref{lemma.D_1(phi) is CM} we can apply the Lemma above to the surjective morphism
$$ \pr_{1*}(\Ocal_{\wt{\Sigma}_1}\otimes\pr_2^*\Ocal(4))|_{D_1(\varphi)}\arr \Lcal. $$
In this way we obtain a locally free sheaf $\Ecal$ defined as the kernel of the morphism
$$ \pr_{1*}(\Ocal_{\wt{\Sigma}_1}\otimes\pr_2^*\Ocal(4))\arr i_*\Lcal. $$
By construction the induced morphism $\varphi':\Ocal_{\wt{H}\times\ptwo}\otimes W_4\arr \Ecal$ is now surjective, thus the kernel $\Fcal:=\ker(\varphi')$ is a locally free subsheaf of $\Ocal_{\wt{H}\times\ptwo}\otimes W_4$, which away from $D_1(\varphi)$ coincides with the kernel of $\varphi$. In this way we have proved Proposition \ref{propos.extension}.

\section{The integral Chow ring of $\Xcal_4$ and $\Mcal_3 \smallsetminus \Hcal_3$}

Consider the equivariant stratification $\ov{Z}_{\rm binod} \subset Z_4$. By Corollary \ref{corol.image in ideal} we know that the image of 
$ A^{\GL_3}_*(\ov{Z}_{\rm binod})\arr A^{\GL_3}_*(\PP(W_4)) $
is contained in $(I_{\wt{Z}},\delta_{\{ 1, 3 \}})$. 

By construction, the image of 
\[ A^{\GL_3}_*(Z_4\setminus \ov{Z}_{\rm binod}) \arr A^{\GL_3}_*(\PP(W_4)\setminus \ov{Z}_{\rm binod}) \]
is contained in the restriction of $I_{\wt{Z}}$ to the latter ring, because $\wt{Z}_4\arr Z_4$ is a Chow envelope over $Z_4\setminus \ov{Z}_{\rm binod}$ (see Definition \ref{defi:IZtilde}).

Therefore, applying Lemma \ref{basic.principle}, we deduce that the image of $A^{\GL_3}_*(Z_4)\arr A^{\GL_3}_*(\PP(W_4))$ is contained in $(I_{\wt{Z}},\delta_{\{1,3\}})$ and hence the two ideals coincide. Thus, we have just proved our first main result:
\begin{theorem}\label{main.result.quartics}
	Assume that the base field has characteristic different from 2 and 3. Then
	\[
	A^* (\Xcal_4) \simeq \ZZ[c_1,c_2,c_3,h_4]/ \left( \alpha_1,\alpha_2,\alpha_3, \delta_{\{ 1,3 \}} \right)
	\]
	where
	\begin{eqnarray*}
		\alpha_1 &=& 27 h_4 - 36 c_1,\\ 
		\alpha_2 &=& 9 h_4^2 - 6 c_1 h_4 - 24 c_2,\\
		\alpha_3 &=& h_4^3 - c_1 h_4^2 + c_2 h_4 - 28 c_3,\\
		\delta_{\{ 1,3 \}} &=& 55 h_4^3 - 220 c_1 h_4^2 + (280 c_1^2 + 40 c_2) h_4 + (224 c_3- 96 c_1^3 - 128 c_1 c_2).
	\end{eqnarray*}
\end{theorem}

Let $\Mcal_3\smallsetminus\Hcal_3$ be the stack of smooth, non-hyperelliptic curves of genus $3$: it is well known that this is a smooth Deligne-Mumford stack.

In \cite[Prop. 3.1.3]{DL18} it is shown that $\Mcal_3\smallsetminus\Hcal_3$ is isomorphic to the quotient stack $[(W_4\otimes\mathbb{D}\smallsetminus \Delta)/\GL_3]$, where $\mathbb{D}$ is the determinant representation of $\GL_3$ and $\Delta$ is the invariant closed subscheme of singular quartics. More explicitly, the action of $\GL_3$ is defined by the following formula:
\[ A\cdot f(X,Y,Z):= \det(A)f(A^{-1}(X,Y,Z)) \]
The equivariant morphism $W_4\otimes\mathbb{D}\smallsetminus \Delta \arr \Pro(W_4)\smallsetminus Z_4$ makes the former scheme into an equivariant $\mathbb{G}_m$-torsor over the latter, whose associated equivariant line bundle is $\Ocal(-1)\otimes\mathbb{D}$. In particular, this implies that the pullback morphism
\[ A_{\GL_3}^*(\Pro(W_4)\smallsetminus Z_4) \arr A_{\GL_3}^*(W_4\otimes\mathbb{D}\smallsetminus \Delta) \]
is surjective with kernel generated by $c_1-h_4$.

Observe that $W_4\otimes\mathbb{D}\smallsetminus \Delta$ is the total space of the $\GL_3$-torsor over $\Mcal_3\smallsetminus \Hcal_3$ associated to the restriction of the Hodge bundle. This means that the equivariant Chern classes $c_i$ correspond to the Chern classes $\lambda_i$ of the Hodge bundle. A straightforward computation implies our main result:
\begin{theorem}\label{thm.AM3}
	Over any field of characteristic $\neq 2,3$ we have:
	\[ A^*(\Mcal_3\smallsetminus \Hcal_3)\simeq \ZZ[\lambda_1,\lambda_2,\lambda_3]/(9 \lambda_1, 6(\lambda_1^2+4\lambda_2), \lambda_1\lambda_2-28\lambda_3, \lambda_1^3+28\lambda_3). \]
\end{theorem}



\begin{thebibliography}{99999999}
\bibliographystyle{alpha}

\bibitem[\textbf{DL18}]{DL18} A.~Di Lorenzo: Picard group of moduli of curves of low genus in positive characteristic. Preprint: arXiv:1812.01913 (2018).

\bibitem[\textbf{DL19}]{DL19} A.~Di Lorenzo: The {C}how ring of the stack of hyperelliptic curves of odd genus. Int. Math. Res. Not. IMRN \textbf{rnz101} (2019), https://doi.org/10.1093/imrn/rnz101

\bibitem[\textbf{Ed-Fu08}]{EF} D.~Edidin, D.~Fulghesu: The integral Chow ring of the stack of at most 1-nodal rational curves. Comm. in Alg. \textbf{36} (2008), no. 2, pp. 581-594.

\bibitem[\textbf{Ed-Fu09}]{EFh} D.~Edidin, D.~Fulghesu: The integral Chow ring of the stack of hyperelliptic curves of even genus; Math. Res. Lett. \textbf{16} (2009), no.1, pp.~27-40.

\bibitem[\textbf{Ed-Gr98}]{EG} D.~Edidin, W.~Graham: Equivariant intersection theory; Inv. Math. \textbf{131} (1998), pp.~595-634.

\bibitem[\textbf{Ed-Gr98}]{EGL} D.~Edidin, W.~Graham: Localization in equivariant intersection theory and the Bott residue formula. Amer. J. Math. \textbf{120} no. 3 (1998), pp.~619--636.

\bibitem[\textbf{El-St87}]{ES87} G.~Ellingsrud and S.A.~Str\o mme: On the homology of the Hilbert scheme of points in the plane; Invent. Math. \textbf{87} (1987), pp.343-352.

\bibitem[\textbf{El-St93}]{ES93} G.~Ellingsrud and S.A.~Str\o mme: Towards the {C}how ring of the {H}ilbert scheme of $\PP^2$; J. Reine Angew. Math. \textbf{441} (1993), pp.33-44.

\bibitem[\textbf{Fab90A}]{m3bar} C.~Faber: Chow rings of moduli spaces of curves. {I}. {T}he {C}how ring of {$\overline{\mathcal M}_3$}; Ann. of Math. \textbf{132}, (1990), 331--419.

\bibitem[\textbf{Fab90B}]{m4} C.~Faber: Chow rings of moduli spaces of curves II: Some results on the Chow ring of $\overline{\mathcal M}_{4}$; Ann. of Math. \textbf{132}, (1990), 421--449.

\bibitem[\textbf{Fog68}]{fog68} J.~Fogarty: Algebraic families on an algebraic surface, Amer. J. Math. \textbf{90}, (1968), 511-521

\bibitem[\textbf{Fu-Vi18}]{Fu-Vi} D.~Fulghesu, A.~Vistoli: The {C}how ring of the stack of smooth plane cubics; Michigan Math. J. \textbf{67} (2018), 3--29.

\bibitem[\textbf{Ful98}]{Ful} W.~Fulton: \emph{Intersection theory.} Second Edition. Springer-Verlag, Berlin, 1998.

\bibitem[\textbf{Lar19}]{la19} E.~Larson: The Integral {C}how ring of $\bar{M}_2$, Preprint: arXiv:1904.08081.

\bibitem[\textbf{Iza95}]{Iza} E.~Izadi, The Chow ring of the moduli space of curves of genus 5; in The Moduli Space of Curves, R. Dijkgraaf, C. Faber, and G. van der Geer (eds), Progress in Math. 129, Birkh\"auser, Boston, 1995.

\bibitem[\textbf{Mar82}]{Mar82}	M.~Maruyama, Elementary transformations in the theory of algebraic vector
		bundles; in Algebraic geometry(La R\'{a}bida, 1981), Lecture Notes in Math.\ 961, Springer, Berlin, 1982.

\bibitem[\textbf{Mum83}]{Mum} D.~Mumford, Towards An Enumerative Geometry of the Moduli Space of Curves, Progress in Math., Vol. II (1983), pp. 271--328.

\bibitem[\textbf{Pe-Vi15}]{PV} N.~Penev and R.~Vakil: The Chow ring of the moduli space of curves of genus six; Algebr. Geom. 2 (2015) \textbf{1}, 123--136. 

\bibitem[\textbf{Sha74}]{Shaf}I.~R.~Shafarevich: \emph{Basic Algebraic Geometry.} Springer-Verlag, Berlin, Heidelberg, New York 1974.

\bibitem[\textbf{Vis98}]{Vis98} A.~Vistoli, {\em The Chow ring of {${\mathcal M}\sb 2$}}, Invent. Math. \textbf{131} (1998), no.~3, 635--644.


\end{thebibliography}
\end{document}